\documentclass[11pt,a4paper]{amsart}
\usepackage[T1]{fontenc}
\usepackage[utf8]{inputenc}
\usepackage{lmodern}
\usepackage{amsmath,amssymb,amsthm}
\usepackage[left=26mm,right=26mm,top=25mm,bottom=25mm,headheight=14pt]{geometry}
\usepackage{microtype}
\usepackage{enumitem}
\usepackage{needspace}
\usepackage{mathtools}
\usepackage{xcolor}
\usepackage[unicode,pdfencoding=auto,colorlinks=true,linkcolor=blue!40!black,citecolor=blue!40!black,urlcolor=blue!40!black]{hyperref}
\hypersetup{pdftitle={Block relations and Clifford theory for H-triples},pdfauthor={Shi Chen},pdfsubject={Polished full manuscript}}
\setlist[enumerate,1]{label=\textup{(\roman*)},ref=\textup{(\roman*)},leftmargin=2.5em,itemsep=3pt,topsep=5pt}
\newcommand{\HH}{\mathcal H}
\newcommand{\PP}{\mathcal P}
\newcommand{\QQ}{\mathcal Q}
\newcommand{\RR}{\mathcal R}
\newcommand{\Qab}{\mathbb Q^{\mathrm{ab}}}
\newcommand{\OO}{\mathcal O}
\DeclareMathOperator{\Irr}{Irr}
\DeclareMathOperator{\Bl}{Bl}
\DeclareMathOperator{\bl}{bl}
\DeclareMathOperator{\Gal}{Gal}
\DeclareMathOperator{\Aut}{Aut}
\DeclareMathOperator{\Inn}{Inn}
\DeclareMathOperator{\GL}{GL}
\DeclareMathOperator{\Sym}{Sym}
\DeclareMathOperator{\Syl}{Syl}
\DeclareMathOperator{\htc}{ht}
\DeclareMathOperator{\tr}{tr}
\DeclareMathOperator{\Br}{Br}
\DeclareMathOperator{\N}{N}
\DeclareMathOperator{\C}{C}
\DeclareMathOperator{\Z}{Z}
\newtheorem{lemma}{Lemma}[section]
\newtheorem{proposition}[lemma]{Proposition}
\newtheorem{theorem}[lemma]{Theorem}
\newtheorem{corollary}[lemma]{Corollary}
\newtheorem*{theoremA}{Theorem A}
\theoremstyle{definition}
\newtheorem{definition}[lemma]{Definition}
\theoremstyle{remark}
\newtheorem*{remark}{Remark}
\numberwithin{equation}{section}
\newcommand{\doi}[1]{\href{https://doi.org/#1}{doi:\,\nolinkurl{#1}}}
\newcommand{\record}[2]{\href{#1}{#2}}
\title[Block relations and Clifford theory]{Block relations and Clifford theory for $\HH$-triples}
\author{Shi Chen}
\address{School of Mathematics and Statistics, Central China Normal University,
Wuhan 430079, P. R. China}
\email{chenshitjnu@163.com}
\date{}
\begin{document}
\begin{abstract}
We study block relations between $\HH$-triples of finite groups.
We prove a Clifford correspondence theorem in which a single pair of
associated projective representations satisfies both the mixed
comparison-function identities and the block conditions on every
intermediate group. The resulting correspondence preserves relative character heights and
yields a common defect group for each pair of corresponding blocks.
We also establish induction and gluing results, including an induction
criterion that allows unequal Clifford indices. Direct products,
wreath products and changes of ambient group are treated within the
same framework. These results provide character-triple tools for
Galois-equivariant local--global correspondences.
\end{abstract}

\maketitle

\section*{Introduction}

Clifford theory relates the irreducible characters of a finite group to
those of its normal subgroups \cite{Clifford,Isaacs}. Character triples
allow these relations to be compared across different groups. In block
theory, the resulting correspondences must also be compatible with block
induction on intermediate groups. This compatibility is a key ingredient
in the work of Navarro and Sp\"ath \cite{NS} and in reduction theorems for
the Alperin--McKay and Dade conjectures \cite{AM,Dade}; see also
\cite{SpConditions}. Related uses of character triples in the McKay
conjecture appear in \cite{IMN,Rossi,NavMcKay}. Results on Clifford
theory, induced blocks and character heights were obtained by
Koshitani and Sp\"ath \cite{KS} and by Murai \cite{Murai94,Murai96}.
For general background on block theory, see \cite{Linck1,Linck2,NavBlocks}.

Galois refinements of local--global conjectures require these
correspondences to respect an additional action on characters
\cite{NavGalois}. Clifford theory with Galois action, including questions
concerning fields of values and Schur indices, was studied by Ladisch
\cite{Ladisch}. Navarro, Sp\"ath and Vallejo \cite{NSV} introduced
$\HH$-triples and relations between them in their reduction of the
Galois--McKay conjecture. Their construction encodes the action of mixed
group--Galois stabilizers through comparison functions of associated
projective representations. For a block relation between $\HH$-triples,
these functions and the correspondences compatible with block induction
must arise from the same pair of projective representations. Establishing
this simultaneous compatibility under Clifford correspondence is the
main purpose of the present paper.

In the modular setting, Feng, Fu and Zhou \cite{FFZ} develop
central and block isomorphisms of $\HH$-triples in their study of the
blockwise Navarro Alperin weight conjecture. Here we work with ordinary
characters and track both relative heights and the specified defect groups.

Fix a prime $p$. Let
$\HH=\HH_p\leq\Gal(\mathbb Q^{\mathrm{ab}}/\mathbb Q)$ consist of the
automorphisms whose action on all roots of unity of order prime to $p$
is a common integral power of $\xi\mapsto\xi^p$.
An $\HH$-triple $(G,N,\theta)_{\HH}$ consists of a normal subgroup
$N\unlhd G$ and a character $\theta\in\Irr(N)$ whose $G$-conjugates
belong to its Galois orbit $\theta^{\HH}$. Definition~\ref{def:block} introduces the
block relation $\geq_b$ by combining the mixed comparison-function
conditions with the ordinary block conditions, including the common
defect group and the correspondences on intermediate groups. For a
character $\chi$ of a normal subgroup of $A$, write $A_{\chi^{\HH}}$
for the stabilizer of its Galois orbit. We use
$\Irr(X\mid\theta^{\HH})$ for the irreducible characters of $X$
lying over a member of $\theta^{\HH}$.

Our main result transfers the block relation to corresponding characters
of normal intermediate subgroups.

\begin{theoremA}
Suppose that
\[
 (G,N,\theta)_{\HH}\geq_b(H,M,\varphi)_{\HH},
 \qquad N\leq J\unlhd G,
\]
and fix a pair of associated projective representations affording this
relation, with common defect group $D$. Set $L=J\cap H$. The
correspondences determined by this pair and its simultaneous Galois
transforms, followed by Clifford induction, give an
$H\times\HH$-equivariant bijection
\[
 \Omega_J:\Irr(J\mid\theta^{\HH})
 \longrightarrow\Irr(L\mid\varphi^{\HH}).
\]
For every $\psi$ in its domain, with $\psi'=\Omega_J(\psi)$, we have
\[
 (G_{\psi^{\HH}},J,\psi)_{\HH}
 \geq_b(H_{(\psi')^{\HH}},L,\psi')_{\HH}.
\]
If $\psi\in\Irr(J\mid\theta)$, then
\[
 \frac{\psi(1)}{\theta(1)}=\frac{\psi'(1)}{\varphi(1)},
 \qquad
 \htc(\psi)-\htc(\theta)=\htc(\psi')-\htc(\varphi),
\]
and the blocks of $\psi$ and $\psi'$ have a common defect group $Q$
such that $Q\cap N=Q\cap M=D$ and $\N_J(Q)\leq L$.
\end{theoremA}

Theorem~\ref{thm:2-19} gives the full statement, including the induced
correspondences on every intermediate group. The normal subgroup $J$
need not be contained in the inertia group $G_\theta$, and the characters
$\theta$ and $\varphi$ may have arbitrary heights. In particular, the
correspondence preserves heights whenever
$\htc(\theta)=\htc(\varphi)$. Corollary~\ref{cor:2-20} specializes the result to
height-zero characters with a specified defect group.

The proof first treats the case of homogeneous restriction. In the
general case, we construct induced projective representations using a
common transversal and compute their comparison functions for the full
mixed stabilizers. We then establish the block-induction identities
for the correspondences afforded by these same representations.
This construction also identifies the maps on intermediate groups and
allows us to compare the defect groups and relative heights of
corresponding characters.

We further consider families of correspondences indexed by characters
of a normal subgroup. Theorem~\ref{thm:3-2} transfers such a family to a larger
normal subgroup, and Theorem~\ref{thm:3-3} assembles correspondences on inertia
groups into one on the whole group. The latter construction allows
the two Clifford indices to differ. The induction criterion in
Lemma~\ref{lem:3-1} supplies the required conditions, including induction
bijections on intermediate groups and corrections within mixed
stabilizers. We also establish structural properties for direct
products, permutations of repeated factors and changes of ambient
group. These include the block butterfly theorem for height-zero
characters in Lemma~\ref{lem:butterfly}.

Section~1 defines the block relation and proves its structural
properties. Section~2 establishes Theorem~A and the relative-height
formula. Section~3 develops the induction and gluing results.

\section{\texorpdfstring{$\HH$}{H}-triples and blocks}

We define a block relation between $\HH$-triples and establish its
compatibility with restriction, composition, products and changes of
ambient group. Throughout, we keep track of the associated projective
representations: the same pair must determine both the mixed
group--Galois comparison functions and the character correspondences
compatible with block induction.

\subsection{Notation and ambient groups}

All groups are finite, and all characters are ordinary characters.
An \emph{overgroup} of $X$ is a group containing $X$ as a subgroup.
An \emph{ambient group} of $X$ is an overgroup $A$ such that
$X\unlhd A$. Thus a statement quantified over all ambient groups of
$X$ concerns every pair $X\unlhd A$ with $A$ finite. For $D\leq X$,
we have
\[
  \N_X(D)\unlhd\N_A(D),
\]
so $\N_A(D)$ is an ambient group for the corresponding local problem.

Fix a prime $p$. Let $\HH=\HH_p\leq\Gal(\Qab/\mathbb Q)$ be the
subgroup consisting of the automorphisms whose action on all roots of
unity of order prime to $p$ is a common integral power of
$\xi\mapsto\xi^p$. This power may depend on the automorphism, but
not on the root of unity. We use right actions. For $X\unlhd A$,
$\chi\in\Irr(X)$, $h\in A$ and $\sigma\in\HH$, we set
\[
  \chi^{h\sigma}(x)=\sigma\bigl(\chi(hxh^{-1})\bigr)
  \qquad(x\in X).
\]
The group and Galois actions commute. For a subgroup $W$, we write
$W^h=h^{-1}Wh$. The inertia group of $\chi$ in $A$ is denoted by
$A_\chi$, and
\[
  A_{\chi^{\HH}}=\{a\in A\mid\chi^a\in\chi^{\HH}\}
\]
denotes the stabilizer in $A$ of the Galois orbit $\chi^{\HH}$.

Write $\Bl(X)$ for the set of $p$-blocks of $X$. For
$b\in\Bl(X)$, let $\Irr_0(b)$ denote the set of irreducible
characters of height zero in $b$. For a $p$-subgroup $D\leq X$, set
\[
  \Irr_0(X\mid D)=
  \bigcup_{\substack{b\in\Bl(X)\\D\text{ is a defect group of }b}}
  \Irr_0(b).
\]
The subgroup $D$ is specified, and we keep this choice throughout the
constructions below. We use Brauer's definition of block induction.
For a block $c$ of a subgroup of $J$, the notation $c^J$ is used only
when the induced block is defined.

For modular reduction, choose a splitting $p$-modular system
$(K,\OO,k)$ large enough for the groups under consideration. The
unique maximal ideal of $\OO$ is its Jacobson radical $J(\OO)$.
Thus $k=\OO/J(\OO)$, and we write $a^*=a+J(\OO)$ for the residue
class of $a\in\OO$.

\subsection{Character fibers}

Let $K\unlhd X$. For $\zeta\in\Irr(K)$, write
$\Irr(X\mid\zeta)$ for the irreducible characters of $X$ whose
restriction to $K$ contains $\zeta$. For $\mathcal S\subseteq\Irr(K)$, set
\[
  \Irr(X\mid \mathcal S)=\bigcup_{\zeta\in \mathcal S}\Irr(X\mid\zeta),
  \qquad
  \Irr_0(X\mid D,\mathcal S)=\Irr_0(X\mid D)\cap\Irr(X\mid \mathcal S).
\]
We omit braces when $\mathcal S$ is a singleton. These sets are called the
\emph{fibers} over the indicated character or set of characters of
$K$. Fibers over individual characters need not be disjoint: by
Clifford theory, $\Irr(X\mid\zeta)=\Irr(X\mid\eta)$ whenever
$\zeta$ and $\eta$ are $X$-conjugate. Accordingly, the
disjoint decompositions used below are indexed by $X$-orbits, or by
$X\times\HH$-orbits when the fibers are taken over Galois orbits.

For $Z\leq\Z(X)$ and $\lambda\in\Irr(Z)$, the central-character
fiber is
\[
  \Irr_0(X\mid D,\lambda)
  =\{\chi\in\Irr_0(X\mid D)\mid\chi_Z=\chi(1)\lambda\}.
\]
As $\lambda$ ranges over $\Irr(Z)$, these fibers form a disjoint
decomposition of $\Irr_0(X\mid D)$. If $X\unlhd A$ and $a\in A$
normalizes both $Z$ and $D$, then $(a,\sigma)$, with $\sigma\in\HH$,
maps the fiber over $\lambda$ onto the fiber over
$\lambda^{a\sigma}$. A bijection preserves central characters on a
common central subgroup if it maps each $\lambda$-fiber onto the
$\lambda$-fiber on the other side. This condition is imposed before
passing to quotients by character kernels.

\subsection{Projective representations and mixed stabilizers}
\label{subsec:projective}

Projective representations are denoted by calligraphic letters. For
$N\unlhd G$ and $\theta\in\Irr(N)$, we call
$(G,N,\theta)_{\HH}$ an $\HH$-triple if
\[
  \{\theta^g\mid g\in G\}\subseteq\theta^{\HH}.
\]
Its corresponding ordinary character triple is $(G_\theta,N,\theta)$.
For $H\leq G$, define the mixed stabilizer by
\[
  (H\times\HH)_\theta
  =\{(h,\sigma)\in H\times\HH\mid\theta^{h\sigma}=\theta\}.
\]
A projective representation $\PP$ of $G_\theta$ is associated with
the ordinary character triple $(G_\theta,N,\theta)$ if its restriction
to $N$ affords $\theta$ and it satisfies
\[
  \PP(nx)=\PP(n)\PP(x),\qquad
  \PP(xn)=\PP(x)\PP(n)
  \qquad(n\in N,\ x\in G_\theta).
\]
Our convention for its factor set is
\[
  \PP(x)\PP(y)=\alpha(x,y)\PP(xy).
\]
The factor set $\alpha$ is therefore inflated from $G_\theta/N$.
We choose the representations over sufficiently large finite
cyclotomic fields and require their factor sets to take values in
roots of unity.

Let $\mathcal A:L\longrightarrow\GL(V)$ and
$\mathcal B:L\longrightarrow\GL(W)$ be ordinary or projective
representations over the same field $F$. An $F$-linear map
$T:V\longrightarrow W$ is an \emph{intertwiner} from $\mathcal A$
to $\mathcal B$ if
\[
  T\mathcal A(x)=\mathcal B(x)T\qquad(x\in L).
\]
Its matrix in chosen bases is called an \emph{intertwining matrix}.
We write $\mathcal A\sim\mathcal B$ if an invertible intertwiner
exists; equivalently, the two representations are conjugate by a single
invertible matrix independent of $x$. All intertwining matrices used
below are invertible.

For $a=(h,\sigma)\in(G\times\HH)_\theta$, the construction in
\cite[Lemma~1.4]{NSV} gives a uniquely determined normalized function
$\mu_a$ and an invertible matrix $T_a$ such that
\begin{equation}\label{eq:comparison}
  \PP(hxh^{-1})^\sigma
  =\mu_a(x)T_a\PP(x)T_a^{-1}
  \qquad(x\in G_\theta).
\end{equation}
Here $T_a$ is an intertwining matrix from $\PP$ to the projective
representation
\[
  x\longmapsto\mu_a(x)^{-1}\PP(hxh^{-1})^\sigma.
\]
The matrix $T_a$ is independent of $x$. The function $\mu_a$ is
constant on $N$-cosets, satisfies $\mu_a(1)=1$, and is unaffected
by scalar rescaling of $T_a$. If
$f:G_\theta/N\longrightarrow(\Qab)^\times$ satisfies $f(1)=1$,
then replacing $\PP$ by $f\PP$ changes the factor set and the
comparison functions according to
\begin{equation}\label{eq:rescaling}
  \alpha\longmapsto\alpha\,df,
  \qquad
  \mu_a(x)\longmapsto
  \mu_a(x)\frac{f(hxh^{-1})^\sigma}{f(x)}.
\end{equation}
Here $f$ is also viewed as an inflated function on $G_\theta$, and
$df(x,y)=f(x)f(y)/f(xy)$. In particular, equality of factor sets
alone does not imply equality of mixed comparison functions.

\subsection{Block relations and a criterion}

The following definition combines the block conditions of \cite{NS}
with the mixed comparison functions of \cite{NSV}.

\begin{definition}\label{def:block}
Let $N\unlhd G$, let $H\leq G$, and set $M=N\cap H$. Suppose
that $\theta\in\Irr(N)$ and $\varphi\in\Irr(M)$ satisfy
\[
  \{\theta^g\mid g\in G\}\subseteq\theta^{\HH},
  \qquad
  \{\varphi^h\mid h\in H\}\subseteq\varphi^{\HH}.
\]
We write
\[
  (G,N,\theta)_{\HH}\geq_b(H,M,\varphi)_{\HH}
\]
if the following conditions hold.
\begin{enumerate}
\item $G=NH$, $\C_G(N)\leq H$, and
      $(H\times\HH)_\theta=(H\times\HH)_\varphi$.
\item There are associated projective representations
      \[
        \PP:G_\theta\longrightarrow\GL_{\theta(1)}(\Qab),
        \qquad
        \PP':H_\varphi\longrightarrow\GL_{\varphi(1)}(\Qab)
      \]
      whose factor sets take values in roots of unity and agree on
      $H_\varphi\times H_\varphi$. For every $c\in\C_G(N)$,
      the matrices $\PP(c)$ and $\PP'(c)$ are scalar matrices with
      the same scalar.
\item For every $a\in(H\times\HH)_\theta$, the comparison functions
      satisfy $\mu'_a=\mu_a|_{H_\varphi}$.
\item The blocks $\bl(\theta)$ and $\bl(\varphi)$ have a common
      defect group $D$ such that $\N_N(D)\leq M$. For every
      $N\leq W\leq G_\theta$, the correspondence
      \[
        \tau_W:\Irr(W\mid\theta)
        \longrightarrow\Irr(W\cap H\mid\varphi)
      \]
      afforded by $(\PP,\PP')$ satisfies
      \[
        \bl\bigl(\tau_W(\xi)\bigr)^W=\bl(\xi)
        \qquad\bigl(\xi\in\Irr(W\mid\theta)\bigr).
      \]
      In particular, each indicated block induction is defined.
\end{enumerate}
\end{definition}

The relation $\geq_c$ is defined by conditions (i)--(iii). The pair
chosen in (ii) must also afford every correspondence required in (iv).
Condition (i) implies $H_\theta=H_\varphi$ and
$G_\theta=NH_\varphi$.

\begin{remark}
The correspondences in Definition~\ref{def:block}(iv) already have
the following covariance property. For $N\leq W\leq G_\theta$ and
$(h,\sigma)\in(H\times\HH)_\theta$,
\[
  \tau_{W^h}(\xi^{h\sigma})
  =\tau_W(\xi)^{h\sigma}
  \qquad\bigl(\xi\in\Irr(W\mid\theta)\bigr).
\]
This follows from conditions (i)--(iii); see \cite[Lemma~1.9(a)]{NSV}
and Lemma~\ref{lem:2-3}. In particular, each $\tau_W$
is equivariant under $\HH_\theta=\HH_\varphi$, where
$\HH_\theta=\{\sigma\in\HH\mid\theta^\sigma=\theta\}$.
The individual fiber $\Irr(W\mid\theta)$ need not be stable under
the whole of $\HH$.
\end{remark}

\begin{lemma}[A criterion for the block relation]\label{lem:criterion}
Let $(G,N,\theta)_{\HH}$ and $(H,M,\varphi)_{\HH}$ be
$\HH$-triples with $H\leq G$, $G=NH$, and $N\cap H=M$.
Suppose that $\bl(\theta)$ and $\bl(\varphi)$ have a common
defect group $D$ satisfying $\N_N(D)\leq M$. For a pair of
associated projective representations $(\PP,\PP')$, the following
conditions are equivalent.
\begin{enumerate}
\item The pair affords the relation in Definition~\ref{def:block}.
\item The pair affords the $\geq_c$-relation and the block isomorphism
      \[
        (G_\theta,N,\theta)\sim_b(H_\varphi,M,\varphi)
      \]
      of ordinary character triples in the sense of
      \cite[Definition~3.6]{NS}.
\item The pair affords the $\geq_c$-relation, satisfies
      $\C_{G_\theta}(D)\leq H_\varphi$, and, for every
      $N\leq J\leq G_\theta$, its induced correspondence
      \[
        \tau_J:\Irr(J\mid\theta)
        \longrightarrow\Irr(J\cap H_\varphi\mid\varphi)
      \]
      satisfies $\bl(\tau_J(\xi))^J=\bl(\xi)$ for every
      $\xi\in\Irr(J\mid\theta)$, with the indicated block
      induction defined.
\end{enumerate}
In each case, $\N_{G_\theta}(D)\leq H_\varphi$, and hence
$\C_J(D)\leq J\cap H_\varphi$ for every $N\leq J\leq G_\theta$.
\end{lemma}

\begin{proof}
Each condition includes the $\geq_c$-relation, so
$G_\theta=NH_\varphi$ and $N\cap H_\varphi=M$. Consequently,
$J=N(J\cap H_\varphi)$ whenever $N\leq J\leq G_\theta$.
Since $\varphi$ is $H_\varphi$-invariant, the Frattini argument
applied to the defect groups of $\bl(\varphi)$ gives
\[
  H_\varphi=M\N_{H_\varphi}(D),
  \qquad
  G_\theta=N\N_{H_\varphi}(D).
\]
For $x\in\N_{G_\theta}(D)$, write $x=nh$ with $n\in N$ and
$h\in\N_{H_\varphi}(D)$. Then $n\in\N_N(D)\leq M$, whence
$x\in H_\varphi$. This proves the normalizer inclusion and both
centralizer inclusions in the statement. The argument uses only the
$\geq_c$-relation and the common-defect hypotheses, and therefore
applies under each of (i)--(iii).

In (i) and (ii), the fixed pair induces the same identification of
quotients and the same character correspondences. In view of the
centralizer inclusion just proved, the definition of a block
isomorphism in \cite[Definition~3.6]{NS} shows that the remaining
requirement is precisely the block equality in
Definition~\ref{def:block}(iv). Thus (i) and (ii) are equivalent.

Finally, $J\leq G_\theta$ implies $J\cap H=J\cap H_\varphi$.
The correspondences in (iii) are therefore exactly those in
Definition~\ref{def:block}(iv), afforded by the same pair. Hence
(i) implies (iii). Conversely, (iii), together with the common-defect
hypotheses, supplies all the conditions of Definition~\ref{def:block}.
\end{proof}

The criterion retains both the specified common defect group and the
block identities on all intermediate groups. It allows us to verify
the ordinary block conditions and the mixed comparison functions
separately for a fixed pair of representations.

\begin{lemma}[Trace criterion]\label{lem:trace}
Under the common-defect hypotheses of Lemma~\ref{lem:criterion},
suppose that $(\PP,\PP')$ affords the $\geq_c$-relation. Then
it affords the $\geq_b$-relation if and only if
\begin{equation}\label{eq:trace}
  \left(\frac{|N|_{p'}\tr\PP(x)}
  {p^{\htc(\theta)}\theta(1)_{p'}}\right)^*
  =
  \left(\frac{|M|_{p'}\tr\PP'(x)}
  {p^{\htc(\varphi)}\varphi(1)_{p'}}\right)^*
\end{equation}
for every $p'$-element $x\in H_\varphi$ such that
$D\in\Syl_p(\C_N(x))$.
\end{lemma}

\begin{proof}
Apply \cite[Theorem~4.4]{NS} to $(G_\theta,N,\theta)$ and
$(H_\varphi,M,\varphi)$ with the given pair. The factor-set and
central-scalar conditions are already satisfied. The integrality
needed to take residues in \eqref{eq:trace} follows from
\cite[Lemma~4.3]{NS}. The theorem characterizes the ordinary block
isomorphism afforded by this pair, so Lemma~\ref{lem:criterion}
gives the assertion.
\end{proof}

\subsection{Transport and composition}

We first record how restriction, Galois conjugation and group
isomorphisms transport the relation and its associated representations.
We then prove transitivity by arranging that the two pairs agree on
the middle triple.

\begin{lemma}\label{lem:galois-blocks}
Let $b\in\Bl(X)$, let $D$ be a defect group of $b$, and let
$\sigma\in\HH$. Then $D$ is a defect group of $b^\sigma$.
Moreover, Brauer correspondence commutes with the action of $\sigma$.
\end{lemma}

\begin{proof}
Write $|X|=p^a m$, where $p\nmid m$, and choose compatible
cyclotomic roots of unity in $\mathbb C$ and
$\overline{\mathbb Q}_p$. We use the splitting $p$-modular system
$(K,\OO,k)$ with $K=\mathbb Q_p(\zeta_{|X|})$, where $\OO$ is
its valuation ring and $k=\OO/J(\OO)$. The extensions
$\mathbb Q_p(\zeta_{p^a})/\mathbb Q_p$ and
$\mathbb Q_p(\zeta_m)/\mathbb Q_p$ are totally ramified and
unramified, respectively, and are linearly disjoint. Since
$\sigma\in\HH$, its action on the $m$th roots of unity is a
power of Frobenius. Together with its action on the $p^a$th roots
of unity, this defines an automorphism of $K$, again denoted by
$\sigma$, agreeing with the original action on character values.
It preserves $\OO$ and $J(\OO)$ and induces an automorphism
$\overline\sigma$ of $k$.

Let $e_b\in\Z(\OO X)$ be the block idempotent of $b$.
Applying $\sigma$ coefficientwise to the ordinary central
idempotents shows that $e_{b^\sigma}=e_b^\sigma$. For every
$p$-subgroup $Q\leq X$, the Brauer map
$\Br_Q^X:(\OO X)^Q\longrightarrow k\C_X(Q)$ satisfies
\[
  \Br_Q^X(e_{b^\sigma})
  =\Br_Q^X(e_b)^{\overline\sigma}.
\]
Indeed, the Brauer map retains the coefficients supported on
$\C_X(Q)$ and reduces them modulo $J(\OO)$, and these operations
commute with the corresponding coefficient actions. Thus the two
Brauer images are nonzero simultaneously for every $Q$. Defect groups
are precisely the maximal $p$-subgroups for which the Brauer image
of the block idempotent is nonzero. Hence $b$ and $b^\sigma$ have
the same defect groups, proving the first assertion.

Set $Y=\N_X(D)$, and let $c$ be the Brauer correspondent of $b$
in $Y$. Since $\C_Y(D)=\C_X(D)$, the idempotent formulation of
Brauer's first main theorem characterizes $c$ as the unique block
of $Y$ with defect group $D$ satisfying
\[
  \Br_D^Y(e_c)=\Br_D^X(e_b).
\]
Applying $\overline\sigma$ to this equality gives
\[
  \Br_D^Y(e_{c^\sigma})
  =\Br_D^Y(e_c)^{\overline\sigma}
  =\Br_D^X(e_b)^{\overline\sigma}
  =\Br_D^X(e_{b^\sigma}).
\]
The first assertion, applied to $X$ and $Y$, shows that both
$b^\sigma$ and $c^\sigma$ have defect group $D$. The uniqueness
of the Brauer correspondent now proves the second assertion.
\end{proof}

\begin{lemma}\label{lem:transport}
Suppose that $(G,N,\theta)_{\HH}\geq_b(H,M,\varphi)_{\HH}$ is
afforded by $(\PP,\PP')$. Each of the following constructions
yields a block relation afforded by the indicated pair.
\begin{enumerate}
\item If $N\leq G_1\leq G$ and $H_1=G_1\cap H$, restrict the
      representations to $(G_1)_\theta$ and $(H_1)_\varphi$.
\item If $\sigma\in\HH$, apply $\sigma$ to the characters and
      to every matrix entry of both representations.
\item Transport the groups, characters, specified defect group and
      both representations along the same group isomorphism.
\end{enumerate}
The resulting character correspondences commute with these
constructions.
\end{lemma}

\begin{proof}
In (i), we have $G_1=NH_1$. Restricting the original data gives the
required identities for mixed stabilizers, factor sets, comparison
functions and central scalars. Every intermediate group for the
restricted relation is already an intermediate group for the original
one. The characters of the normal subgroups and their common defect
group $D$ remain unchanged.

In (ii), Galois action commutes with group conjugation and with the
tensor construction of the character correspondences from projective
representations of the quotients. Applying $\sigma$ to the defining
matrix equations gives the factor sets and comparison functions for
the conjugate pair. Lemma~\ref{lem:galois-blocks} shows that $D$ is
still a common defect group. The class-sum definition of block
induction is compatible with the induced coefficient action, so the
block-induction identities pass to the conjugate blocks.

For (iii), transport the matrix equations and the class-sum identities
along the isomorphism. The tensor construction of the correspondences
commutes with this transport as well.
\end{proof}

\begin{lemma}\label{lem:transitivity}
Suppose that
\[
  (G,N,\theta)_{\HH}\geq_b(H,M,\varphi)_{\HH}
  \geq_b(J,L,\psi)_{\HH},
\]
and that both relations are given with the same defect group $D$.
Then $(G,N,\theta)_{\HH}\geq_b(J,L,\psi)_{\HH}$.
\end{lemma}

\begin{proof}
Choose pairs $(\PP,\PP')$ and $(\RR,\RR')$ affording the two
relations. We first make the representations associated with the
middle triple agree. Conjugating $\RR$ by a fixed invertible matrix,
we may assume that $\RR_M=\PP'_M$; no change of basis in $\RR'$
is needed. Schur's lemma then gives
\[
  \RR(h)=c(h)\PP'(h)\qquad(h\in H_\varphi),
\]
where $c(1)=1$ and $c$ is constant on $M$-cosets. Since both factor sets take values in roots of unity, the matrices
$\RR(h)$ and $\PP'(h)$ have finite order for every $h\in H_\varphi$.
Raising the displayed equality to a common multiple of their orders
shows that $c(h)$ is a root of unity.

Rescale $\RR$ by $c^{-1}$ and $\RR'$ by the restriction of
$c^{-1}$. We then have $\RR=\PP'$. Both factor sets are
multiplied by $d(c^{-1})$, and the comparison functions acquire
the same factor $(c^{-1})^a/c^{-1}$, restricted to the appropriate
domain. The block correspondences are unchanged: a quotient
representation $\QQ$ whose factor set cancels that of the rescaled
pair corresponds to $\QQ c^{-1}$ for the original pair, and the
resulting tensor products are identical.

Thus $(\PP,\RR')$ has matching factor sets and mixed comparison
functions. The group conditions follow from
\[
  G=NH=N(MJ)=NJ,\qquad N\cap J=L.
\]
If $c\in\C_G(N)$, then $c\in H$ and $c$ centralizes $M$, so
$c\in J$. The original central-scalar identities give the required
scalar equality for $\PP(c)$ and $\RR'(c)$. The mixed stabilizer
equalities likewise compose. If $g\in\N_N(D)$, the first relation
gives $g\in M$, and the second then gives $g\in L$.

It remains to verify block induction. Let $N\leq W\leq G_\theta$
and choose a projective representation $\QQ$ of $W/N$ whose
inflated factor set cancels that of $\PP_W$. Using
\[
  W/N\cong(W\cap H)/M,\qquad M\leq W\cap H\leq H_\varphi,
\]
the two correspondences send $\tr(\QQ\otimes\PP)$ first to
$\tr(\QQ\otimes\PP')$ and then to $\tr(\QQ\otimes\RR')$,
with the appropriate restrictions understood. Since the middle
representations agree, their composite is precisely the correspondence
afforded by $(\PP,\RR')$. Transitivity of block induction now
gives the required block equality.
\end{proof}

\subsection{Direct products and wreath products}

For characters of different direct factors, $\otimes$ denotes the
external tensor product. A Galois automorphism acts simultaneously
on all factors, so we take the stabilizer of the $\HH$-orbit of the
tensor product as the ambient group. We first treat direct products
and then allow permutations of the factors.

\begin{lemma}[Direct products]\label{lem:products}
For $i=1,2$, suppose that
\[
  (G_i,N_i,\theta_i)_{\HH}
  \geq_b(H_i,M_i,\varphi_i)_{\HH}
\]
is afforded by $(\PP_i,\PP'_i)$ with common defect group $D_i$.
Set
\[
\begin{aligned}
  G&=G_1\times G_2,& H&=H_1\times H_2,\\
  N&=N_1\times N_2,& M&=M_1\times M_2,\\
  \Theta&=\theta_1\otimes\theta_2,&
  \Phi&=\varphi_1\otimes\varphi_2,
\end{aligned}
\qquad D=D_1\times D_2.
\]
Then $(\PP_1\otimes\PP_2,\PP'_1\otimes\PP'_2)$ affords
\[
  (G_{\Theta^{\HH}},N,\Theta)_{\HH}
  \geq_b(H_{\Phi^{\HH}},M,\Phi)_{\HH},
\]
with common defect group $D$.
\end{lemma}

\begin{proof}
Write $G^\times=G_{\Theta^{\HH}}$ and
$H^\times=H_{\Phi^{\HH}}$. For $g=(g_1,g_2)\in G^\times$,
choose a single $\sigma\in\HH$ such that
$\theta_i^{g_i}=\theta_i^\sigma$ for both $i$. Write
$g_i=n_i h_i$ with $n_i\in N_i$ and $h_i\in H_i$. The original
mixed stabilizer equalities imply
$\varphi_i^{h_i}=\varphi_i^\sigma$. Hence
$G^\times=NH^\times$, $N\cap H^\times=M$, and the two mixed
stabilizers agree. The centralizer condition follows coordinatewise.
Moreover,
\[
  (G^\times)_\Theta=(G_1)_{\theta_1}\times(G_2)_{\theta_2},
  \qquad
  (H^\times)_\Phi=(H_1)_{\varphi_1}\times(H_2)_{\varphi_2}.
\]

For $\PP=\PP_1\otimes\PP_2$, the factor set is
\[
  \alpha\bigl((x_1,x_2),(y_1,y_2)\bigr)
  =\alpha_1(x_1,y_1)\alpha_2(x_2,y_2).
\]
If $a=((h_1,h_2),\sigma)$ belongs to the mixed stabilizer, let
$T_{i,(h_i,\sigma)}$ be an intertwining matrix in the comparison
identity for $\PP_i$. Their tensor product
$T_{1,(h_1,\sigma)}\otimes T_{2,(h_2,\sigma)}$ gives
\[
  \mu_a(x_1,x_2)
  =\mu_{1,(h_1,\sigma)}(x_1)\mu_{2,(h_2,\sigma)}(x_2).
\]
The corresponding calculation for the primed representations yields
the same function on the local inertia group, as in
\cite[Lemma~2.5]{NSV}.

The product blocks have common defect group $D$, and
\[
  \N_N(D)=\N_{N_1}(D_1)\times\N_{N_2}(D_2)\leq M.
\]
By Lemma~\ref{lem:criterion}, each original pair affords an ordinary
block isomorphism. Its proof also supplies the centralizer condition
for the defect group required in the $N_i$-block formulation.
The tensor construction in the proof of \cite[Theorem~5.1]{Dade}
therefore gives the block-induction identity on every intermediate
subgroup of $(G^\times)_\Theta$, including those that are not
direct products. Lemma~\ref{lem:criterion} now proves the assertion.
\end{proof}

\begin{lemma}[Repeated factors]\label{lem:repeated}
Suppose that $(G,N,\theta)_{\HH}\geq_b(H,M,\varphi)_{\HH}$ is
afforded by $(\PP,\PP')$ with common defect group $D$. Let
$r\geq1$ be an integer, and let $\Delta_r$ denote the diagonal
embedding into an $r$-fold direct product. Set
\[
  \widehat G=(G_\theta\wr\Sym(r))\Delta_rG,
  \qquad
  \widehat H=(H_\varphi\wr\Sym(r))\Delta_rH.
\]
Then
\[
  (\widehat G,N^r,\theta^{\otimes r})_{\HH}
  \geq_b(\widehat H,M^r,\varphi^{\otimes r})_{\HH},
\]
with common defect group $D^r$. This relation is afforded by the
tensor-permutation representations constructed from $(\PP,\PP')$.
\end{lemma}

\begin{proof}
The equality
$(\theta^{\otimes r})^{(g_1,\ldots,g_r)}
=(\theta^{\otimes r})^\sigma$ holds precisely when
$\theta^{g_i}=\theta^\sigma$ for every $i$. Thus all $g_i$ lie
in the same $G_\theta$-coset, and
\[
  \widehat G=(G\wr\Sym(r))_{(\theta^{\otimes r})^{\HH}},
  \qquad
  \widehat G_{\theta^{\otimes r}}=G_\theta\wr\Sym(r).
\]
The analogous equalities hold on the local side. The original mixed
stabilizer equality gives the required group factorization and mixed
stabilizer equality, as in \cite[Lemma~2.6]{NSV}.

Let $T_\pi$ permute the tensor factors according to
$\pi\in\Sym(r)$, and define
\begin{equation}\label{eq:wreath}
  \PP^\wr\bigl((x_1,\ldots,x_r)\pi\bigr)
  =\bigl(\PP(x_1)\otimes\cdots\otimes\PP(x_r)\bigr)T_\pi.
\end{equation}
Define $(\PP')^\wr$ similarly. Multiplication in
\eqref{eq:wreath} permutes and multiplies the coordinate factor
sets in the same way on both sides. The resulting factor sets
therefore agree on the local inertia group. The tensor construction
also gives the required central-scalar identities.

For a diagonal mixed element $(\Delta_rh,\sigma)$ stabilizing
$\theta^{\otimes r}$, equation~\eqref{eq:comparison} gives
\[
  \mu^\wr_{(\Delta_rh,\sigma)}
  \bigl((x_1,\ldots,x_r)\pi\bigr)
  =\prod_{i=1}^r\mu_{(h,\sigma)}(x_i).
\]
Indeed, $T_\pi$ has rational entries and commutes with
$T_{(h,\sigma)}^{\otimes r}$. Repeating the calculation for the
wreath-product representation constructed from $\PP'$ gives the
same formula with $\mu'$ in place of $\mu$. Every element of
the local mixed stabilizer is a product of an element of the local
inertia group and a diagonal mixed element of this form. For inertia
elements, the comparison functions are determined by the common
factor set. Composing the comparison equations therefore gives
equality for every mixed element, as in the final calculation of
\cite[Lemma~2.6]{NSV}.

The blocks of the given characters have common defect group $D^r$,
and
\[
  \N_{N^r}(D^r)=\N_N(D)^r\leq M^r.
\]
The block calculation in \cite[Theorem~5.2]{Dade} applies to the
representations in \eqref{eq:wreath} and yields the block-induction
identity on every intermediate group of $G_\theta\wr\Sym(r)$.
Together with the comparison-function equality established above,
Lemma~\ref{lem:criterion} gives the assertion.
\end{proof}

\begin{proposition}[Galois-conjugate factors]\label{prop:galois-factors}
Assume the hypotheses of Lemma~\ref{lem:repeated}. Let $k,r\geq1$
be integers, and let $\sigma_1,\ldots,\sigma_k\in\HH$. Set
$\theta_i=\theta^{\sigma_i}$, and suppose that the $\theta_i$
are pairwise nonconjugate under $G$. Write $s=kr$,
$\varphi_i=\varphi^{\sigma_i}$, and
\[
  \Theta=\bigotimes_{i=1}^k\theta_i^{\otimes r},
  \qquad
  \Phi=\bigotimes_{i=1}^k\varphi_i^{\otimes r}.
\]
Then
\[
  \bigl((G\wr\Sym(s))_{\Theta^{\HH}},N^s,\Theta\bigr)_{\HH}
  \geq_b
  \bigl((H\wr\Sym(s))_{\Phi^{\HH}},M^s,\Phi\bigr)_{\HH}.
\]
The relation is afforded by the tensor product of the representations
obtained by applying \eqref{eq:wreath} to $\PP^{\sigma_i}$ and
$(\PP')^{\sigma_i}$.
\end{proposition}

\begin{proof}
The original mixed stabilizer equality implies that the $\varphi_i$
are pairwise nonconjugate under $H$. The two inertia groups are
therefore
\[
  \prod_{i=1}^k(G_\theta\wr\Sym(r)),
  \qquad
  \prod_{i=1}^k(H_\varphi\wr\Sym(r)).
\]
The result in \cite[Theorem~2.7]{NSV} applies to the representations specified
in the statement. Its comparison calculation includes mixed elements
that permute the $k$ sets of $r$ factors. After factoring out an
inertia element, the intertwining matrix is obtained by composing
the permutation of the $k$ tensor blocks with the tensor product of
the intertwining matrices for the diagonal mixed actions on those
blocks. The resulting scalar is a product of Galois conjugates
of the comparison functions in Lemma~\ref{lem:repeated}. The
corresponding local functions yield the same product. Hence these
representations afford the $\geq_c$-relation, including the
comparison-function identities for mixed elements that permute
distinct Galois-conjugate types.

By Lemma~\ref{lem:galois-blocks}, every character factor has defect
group $D$. Thus the product blocks have common defect group $D^s$,
with $\N_{N^s}(D^s)\leq M^s$. Applying the block conclusions of
Lemmas~\ref{lem:repeated} and~\ref{lem:products} to the displayed
inertia groups gives the block-induction identity for the resulting
correspondences on every intermediate group. Lemma~\ref{lem:criterion}
completes the proof.
\end{proof}

\subsection{The block butterfly theorem}

The next result permits a change of ambient group when the induced
automorphism group of the normal subgroup is unchanged. We use the
construction in \cite[Theorem~2.9]{NSV} for the $\geq_c$-relation
and verify the block-induction identity for the resulting
correspondences on every intermediate group.

\begin{lemma}[Block butterfly theorem]\label{lem:butterfly}
Suppose that
\[
  (G,N,\theta)_{\HH}\geq_b(H,M,\varphi)_{\HH},
\]
where $\theta$ and $\varphi$ have height zero and their blocks
have a common defect group $D$. Let $N\unlhd\widehat G$, and let
\[
  \epsilon:G\longrightarrow\Aut(N),
  \qquad
  \widehat\epsilon:\widehat G\longrightarrow\Aut(N)
\]
be the conjugation homomorphisms. Suppose that
$\epsilon(G)=\widehat\epsilon(\widehat G)$, and set
\[
  \widehat H=\widehat\epsilon^{-1}\bigl(\epsilon(H)\bigr).
\]
Then $\widehat H\cap N=M$ and
\[
  (\widehat G,N,\theta)_{\HH}
  \geq_b(\widehat H,M,\varphi)_{\HH}.
\]
The relation is afforded by the pair of projective representations
obtained from the butterfly construction.
\end{lemma}

\begin{proof}
Since $\C_G(N)\leq H$, we have
$\epsilon(H)\cap\Inn(N)=\epsilon(M)$. It follows that
\[
  \widehat H\cap N=M,\qquad
  \widehat G=N\widehat H,\qquad
  \C_{\widehat G}(N)\leq\widehat H.
\]
The required equality of mixed stabilizers follows because the
actions on the characters of the normal subgroups depend only on
the automorphisms induced on $N$.

Let $(\PP,\PP')$ afford the original relation. Choose a transversal
$T$ for $M\C_G(N)$ in $H_\varphi$ and a corresponding transversal
$\widehat T$ for $M\C_{\widehat G}(N)$ in $\widehat H_\varphi$,
both containing the identity, such that
\[
  \epsilon(t)=\widehat\epsilon(\widehat t)
  \qquad(t\in T,\ \widehat t\in\widehat T\text{ corresponding}).
\]
Let $\lambda$ be the common scalar function of $\PP$ and $\PP'$
on $\C_G(N)$, and let $\omega_\theta$ be the central character
of $\theta$ on $\Z(N)$. Choose a function $\widehat\lambda$ on
$\C_{\widehat G}(N)$, with values in roots of unity, such that
$\widehat\lambda(1)=1$ and
\[
  \widehat\lambda(zc)=\omega_\theta(z)\widehat\lambda(c)
  \qquad(z\in\Z(N),\ c\in\C_{\widehat G}(N)).
\]
Such a function is obtained by prescribing values on representatives
of $\C_{\widehat G}(N)/\Z(N)$ and extending by this identity.
Define
\[
\begin{aligned}
  \widehat\PP(\widehat t n\widehat c)
    &=\PP(t)\PP(n)\widehat\lambda(\widehat c),\\
  \widehat\PP'(\widehat t m\widehat c)
    &=\PP'(t)\PP'(m)\widehat\lambda(\widehat c),
\end{aligned}
\]
where $n\in N$, $m\in M$, and
$\widehat c\in\C_{\widehat G}(N)$. The factorization
$n\widehat c$ is unique up to multiplication by an element of
$\Z(N)$, and the defining identity for $\widehat\lambda$ makes
the first expression independent of this choice. The same argument
applies to the second expression, using the common central scalars.
Multiplication with respect to the chosen transversals gives
associated projective representations whose factor sets agree on
$\widehat H_\varphi$ and whose central scalars coincide, as in
\cite[Theorem~2.9]{NSV}.

We next compare the mixed comparison functions. If
$x\in\widehat H_\varphi$ and $y\in H_\varphi$ induce the same
automorphism of $N$, there is a common root of unity $u(x,y)$ such
that
\begin{equation}\label{eq:butterfly-scalar}
  \widehat\PP(x)=u(x,y)\PP(y),
  \qquad
  \widehat\PP'(x)=u(x,y)\PP'(y).
\end{equation}
Indeed, write $x=\widehat t m\widehat c$. Then $y=tmc$ for some
$c\in\C_G(N)$. If $\alpha$ denotes the original common factor
set on $H_\varphi$, both identities hold with
\[
  u(x,y)=\alpha(tm,c)\widehat\lambda(\widehat c)\lambda(c)^{-1}.
\]
All three factors are roots of unity. For $\lambda(c)$, this
follows from the finite order of $\PP(c)$, which in turn follows
from the root-of-unity values of the factor set. No further
normalization of $\lambda$ is needed.

Let $(\widehat h,\sigma)\in(\widehat H\times\HH)_\theta$,
and choose $h\in H$ such that
$\epsilon(h)=\widehat\epsilon(\widehat h)$. Then
$(h,\sigma)$ belongs to the original mixed stabilizer. For $x,y$
as above, $\widehat h x\widehat h^{-1}$ and $hyh^{-1}$ again
induce the same automorphism of $N$. On each side, the original
intertwining matrix can be used, since the restricted representation and
the induced action on it are unchanged. Substituting
\eqref{eq:butterfly-scalar} into \eqref{eq:comparison} gives
\begin{equation}\label{eq:butterfly-comparison}
  \widehat\mu_{(\widehat h,\sigma)}(x)
  =\frac{u(\widehat h x\widehat h^{-1},hyh^{-1})^\sigma}{u(x,y)}
  \mu_{(h,\sigma)}(y).
\end{equation}
Repeating the calculation for $\PP'$ and its transferred
representation gives the same scalar quotient, with $\mu'$
in place of $\mu$. Thus the two transferred comparison functions
agree on $\widehat H_\varphi$ for every mixed stabilizer element.
This establishes the $\geq_c$-relation.

It remains to prove compatibility with block induction. Let
$x\in\widehat H_\varphi$ be a $p'$-element such that
$D\in\Syl_p(\C_N(x))$. Then $a=\widehat\epsilon(x)$ is a
$p'$-element. Choose $y_0\in H_\varphi$ mapping to $a$.
The $p$-part of $y_0$ maps to the identity, so
$y=(y_0)_{p'}$ is a $p'$-element with $\epsilon(y)=a$.
Consequently,
\[
  \C_N(y)=\C_N(x),\qquad D\in\Syl_p(\C_N(y)).
\]
Since $\theta$ and $\varphi$ have height zero,
Lemma~\ref{lem:trace}, applied to the original pair, gives
\[
  \left(\frac{|N|_{p'}\tr\PP(y)}{\theta(1)_{p'}}\right)^*
  =
  \left(\frac{|M|_{p'}\tr\PP'(y)}{\varphi(1)_{p'}}\right)^*.
\]
Multiplying by the common nonzero residue $u(x,y)^*$ yields the
required trace equality for $\widehat\PP$ and $\widehat\PP'$
at $x$. A second application of Lemma~\ref{lem:trace} gives the
block-induction identity for the associated correspondences on
every intermediate group. The characters of the normal subgroups,
their common defect group $D$, and the inclusion $\N_N(D)\leq M$
remain unchanged.
\end{proof}

\section{Clifford correspondence and relative heights}\label{sec:clifford}

We use the conventions of Section~1. Given a finite collection of
projective representations, we work with the image of $\HH$ in the
Galois group of a sufficiently large finite cyclotomic field containing
their matrix entries and factor-set values. We write
$\PP_1\sim\PP_2$ when there is an invertible intertwining matrix
between them, as in Section~\ref{subsec:projective}.

Throughout this section, the block relation is the one in
Definition~\ref{def:block}. In particular, a single pair of associated
projective representations affords both the mixed comparison functions
and the character correspondences compatible with block induction on
every intermediate group. The corresponding ordinary inertia triples
are therefore block isomorphic in the sense of
\cite[Definition~3.6]{NS}.

\subsection{Inertia groups and Galois orbits}\label{subsec:inertia}

Suppose that
\[
  (G,N,\theta)_{\HH}\geq_b(H,M,\varphi)_{\HH},
  \qquad N\leq J\unlhd G.
\]
Fix a pair $(\PP,\PP')$ affording this relation, the corresponding
family of maps $\tau$, and a common defect group $D$ as in
Definition~\ref{def:block}. Set
\[
  L=J\cap H,\qquad T=J_\theta,\qquad
  U=T\cap H=L_\varphi.
\]
For $\psi\in\Irr(J\mid\theta)$, let
$\eta\in\Irr(T\mid\theta)$ be its Clifford correspondent, and set
$\eta'=\tau_T(\eta)$ and $\psi'=(\eta')^L$.
The resulting transfer theorem is stated in Subsection~2.5, after the
representation and block calculations needed for its proof.

\begin{lemma}[Galois action and defect groups]\label{lem:2-1}
Let $\sigma\in\HH$, let $E$ be a finite group, and let
$b\in\Bl(E)$. Then $b^\sigma$ and $b$ have the same defect groups.
The action of $\sigma$ commutes with every instance of Brauer block
induction used in this section.
\end{lemma}

\begin{proof}
Choose a sufficiently large splitting $p$-modular system
$(K,\OO,k)$ on which $\sigma$ acts, and write $\bar\sigma$ for the
induced automorphism of $k=\OO/J(\OO)$. Such a choice is possible
because $\sigma\in\HH$. The action on coefficients fixes the group
basis and sends the block idempotent $e_b\in\Z(\OO E)$ to
$e_{b^\sigma}$. Thus, for every $p$-subgroup $R\leq E$, the Brauer
map satisfies
\[
  \Br_R(e_{b^\sigma})=\bar\sigma\bigl(\Br_R(e_b)\bigr).
\]
The two Brauer images are nonzero simultaneously. The characterization
of defect groups as maximal subgroups with nonzero Brauer image proves
the first assertion; see also \cite[Section~4, pp.~34--35]{BKY}.

The class-sum projection in the central-character definition of
Brauer block induction also commutes with the action on coefficients.
Applying $\bar\sigma$ to the defining central-character identities
therefore gives the corresponding identities for the Galois-conjugate
blocks. This proves the claimed compatibility whenever block induction
is defined.
\end{proof}

We first establish the group factorizations and character bijections.
The comparison functions for the induced representations are computed
in Subsection~2.3.

\begin{lemma}\label{lem:2-2}
We have $G_\theta\unlhd G$ and $H_\varphi=H_\theta\unlhd H$.
Consequently, $T\unlhd G$ and $U\unlhd H$. Moreover,
\[
  J=NL,\qquad T=NU,\qquad J/T\cong L/U,\qquad
  [J:T]=[L:U].
\]
\end{lemma}

\begin{proof}
For $g\in G$, the $\HH$-triple condition gives $\sigma\in\HH$
such that $\theta^g=\theta^\sigma$. Since conjugation commutes with
the Galois action,
\[
  (G_\theta)^g=G_{\theta^g}=G_{\theta^\sigma}=G_\theta.
\]
The same argument applies to $H_\varphi$, while equality of the mixed
stabilizers of $\theta$ and $\varphi$ gives $H_\theta=H_\varphi$.
Hence $T=J\cap G_\theta\unlhd G$ and
$U=L\cap H_\varphi\unlhd H$. Since $G=NH$ and $N\leq J$,
we have $J=N(J\cap H)=NL$. Inner automorphisms induced by $N$
fix $\theta$, so $T=NL_\theta=NU$. The remaining assertions follow
from the second isomorphism theorem.
\end{proof}

\begin{lemma}[Covariance over a fixed constituent]\label{lem:2-3}
If $N\leq W\leq G_\theta$ and
$a=(h,\sigma)\in(H\times\HH)_\theta$, then
\[
  \tau_{W^h}(\zeta^a)=\tau_W(\zeta)^a
  \qquad(\zeta\in\Irr(W\mid\theta)).
\]
In particular, this holds for $W=T$, where $T^h=T$.
\end{lemma}

\begin{proof}
Write $\zeta=\tr(\QQ\otimes\PP_W)$, where $\QQ$ is inflated
from $W/N$ and has factor set $\alpha^{-1}|_{W\times W}$.
Since $\PP^a\sim\mu_a\PP$ and
$(\PP')^a\sim(\mu_a|_{H_\varphi})\PP'$, the transformed
characters are afforded by
\[
  (\mu_a\QQ^a)\otimes\PP_{W^h},\qquad
  (\mu_a\QQ^a)|_{W^h\cap H}\otimes\PP'_{W^h\cap H},
\]
respectively. These are precisely the representations defining the
correspondence on $W^h$ afforded by the fixed pair $(\PP,\PP')$.
This is the tensor calculation in \cite[Lemma~1.9(a)]{NSV}.
\end{proof}

\begin{lemma}[Correction of mixed stabilizers]\label{lem:2-4}
Write
\[
  \Gamma_\psi=(H\times\HH)_\psi,\qquad
  \Gamma_\eta=(H\times\HH)_\eta.
\]
Then
\begin{align}
  (H\times\HH)_\eta&=(H\times\HH)_{\eta'},\label{eq:2-1}\\
  (H\times\HH)_\psi&=(H\times\HH)_{\psi'},\label{eq:2-2}\\
  \Gamma_\psi&=\Gamma_\eta(L\times1)
                 =(L\times1)\Gamma_\eta.\label{eq:2-3}
\end{align}
More precisely, for every $a\in\Gamma_\psi$, there is $l\in L$
such that $a_0=a(l^{-1},1)\in\Gamma_\eta$; this element fixes
$\theta$, $\varphi$, $\eta$, and $\eta'$ simultaneously.
\end{lemma}

\begin{proof}
The restrictions satisfy $\eta_N=e\theta$ and $\eta'_M=e\varphi$
for the same positive integer $e$. Thus any mixed element fixing
$\eta$ also fixes $\theta$. Equality of the mixed stabilizers of
$\theta$ and $\varphi$, together with Lemma~\ref{lem:2-3}, shows
that it fixes $\eta'$. Conversely, an element fixing $\eta'$ fixes
$\varphi$, hence $\theta$, and injectivity of $\tau_T$ then shows
that it fixes $\eta$. This proves \eqref{eq:2-1}.

Let $a\in\Gamma_\psi$. By Clifford theory, the distinct irreducible
constituents of $\psi_N$ are the characters in
$\{\theta^j\mid j\in J\}$. Hence $\theta^a=\theta^j$ for some
$j\in J$. Write $j=nl$ with $n\in N$ and $l\in L$; then
$\theta^a=\theta^l$. The element $a_0=a(l^{-1},1)$ fixes
$\theta$ and, since $l\in J$, also fixes $\psi$. As $T\unlhd G$,
the character $\eta^{a_0}$ is the Clifford correspondent of $\psi$
over $\theta$. Uniqueness gives $\eta^{a_0}=\eta$. Equation
\eqref{eq:2-1} now gives $(\eta')^{a_0}=\eta'$, and induction yields
$(\psi')^{a_0}=\psi'$. Inner automorphisms induced by $L$ fix
$\psi'$, so $(\psi')^a=\psi'$.

Conversely, suppose that $(\psi')^a=\psi'$. Choose $l\in L$ such
that $\varphi^a=\varphi^l$. Then $a_0=a(l^{-1},1)$ fixes
$\varphi$ and $\psi'$, and hence their Clifford correspondent
$\eta'$. By \eqref{eq:2-1}, it also fixes $\eta$, so $a$ fixes
$\psi$. This proves \eqref{eq:2-2}. The correction just constructed
gives $\Gamma_\psi=\Gamma_\eta(L\times1)$. Since $L\unlhd H$,
the product may also be written in the reverse order.
\end{proof}

\begin{lemma}[Groups in the induction step]\label{lem:2-5}
Set
\[
  K=G_\psi,\qquad K_1=G_\eta,\qquad
  V=H_{\psi'},\qquad V_1=H_{\eta'}.
\]
Then
\begin{align}
  K_1&=G_{\psi,\theta},&
  V_1&=H_{\psi',\varphi},&
  K_1\cap H&=V_1,& K\cap H&=V,\label{eq:2-4}\\
  K&=JK_1=NV,& V&=LV_1,&
  K_1\cap J&=T,& V_1\cap L&=U.\label{eq:2-5}
\end{align}
Thus a right transversal for $U$ in $L$ is also a right transversal
for $K_1$ in $K$ and for $V_1$ in $V$.
If $X=G_{\psi^{\HH}}$ and $Y=H_{(\psi')^{\HH}}$, then
\[
  X=NY=JY,\qquad X\cap H=Y,\qquad
  J\cap Y=L,\qquad \C_X(J)\leq V.
\]
\end{lemma}

\begin{proof}
An element fixing $\eta$ fixes $\eta^J=\psi$ and
$\eta_N/e=\theta$. Conversely, an element fixing $\psi$ and
$\theta$ fixes $\eta$ by uniqueness of the Clifford correspondent.
Thus $K_1=G_{\psi,\theta}$, and the same argument applies locally.
Taking trivial Galois components in \eqref{eq:2-1} and
\eqref{eq:2-2} gives the asserted intersections with $H$.

For $g\in K$, the character $\theta^g$ is a constituent of
$\psi_N$. Multiplying $g$ by a suitable element of $J$ therefore
produces an element fixing $\theta$, so $K=JK_1$. The local argument
gives $V=LV_1$. Also, $K_1\cap J=J_\theta=T$ and $V_1\cap L=U$.
Since $G=NH$ and $N\leq J\leq K$, we have
$K=N(K\cap H)=NV$. Together with $J=NL$ and $N\leq K_1$,
these equalities give
\[
  K=K_1L,\qquad V=V_1L,\qquad K_1\cap L=V_1\cap L=U,
\]
which proves the transversal assertion.

Similarly, $X=N(X\cap H)$, and \eqref{eq:2-2} gives $X\cap H=Y$.
Inner automorphisms induced by $J$ fix $\psi$, and those induced by
$L$ fix $\psi'$. Hence $J\cap Y=L$ and $X=JY$. Finally,
$\C_X(J)\leq\C_G(N)\leq H$, and every element of $\C_X(J)$
fixes $\psi$. Thus $\C_X(J)\leq V$.
\end{proof}

\begin{proposition}[Gluing the character fibers]\label{prop:2-6}
The assertion in Theorem~\ref{thm:2-19}\textup{(a)} holds.
\end{proposition}

\begin{proof}
The groups $T=J_\theta$ and $U=L_\varphi$ are the relevant inertia
groups. Clifford induction on each side and the map $\tau_T$ are
bijections. Their composite
$\eta^J\mapsto\tau_T(\eta)^L$ is therefore a bijection between
the specified fibers.

Set $\Theta=\theta^{\HH}$ and $\Phi=\varphi^{\HH}$. Equality of
the mixed stabilizers of $\theta$ and $\varphi$ shows that
\[
  \theta^{h\sigma}\longmapsto\varphi^{h\sigma}
\]
defines an $H\times\HH$-equivariant bijection
$\Theta\to\Phi$. Since $J=NL$, the $J$-orbits on $\Theta$ are
its $L$-orbits, and this bijection matches them with the $L$-orbits on
$\Phi$. Choose a representative from each $L$-orbit on $\Theta$, transport
$(\PP,\PP')$ simultaneously to the corresponding pair of
constituents, and apply the Clifford construction above.

We verify that the resulting maps are independent of these choices.
If $a=(h,\sigma)$ sends a chosen constituent to another one, use
$(\PP^a,(\PP')^a)$ on the two sides; the tensor correspondences
are then transformed simultaneously by $a$. Two choices leading to
the same constituent differ by an element of its mixed stabilizer,
so Lemma~\ref{lem:2-3} shows that they give the same correspondence.
If the constituents differ by $l\in L$, their Clifford correspondents
are $\eta^l$ and $(\eta')^l$, whose inductions give the same
$J$-character and $L$-character, respectively. Thus the maps do not
depend on the orbit representatives.

By Clifford theory, the fibers belonging to distinct constituent
orbits are disjoint and exhaust the required character sets. The maps
therefore combine to a bijection $\Omega_J$. The construction by
simultaneous transforms makes $\Omega_J$ $H\times\HH$-equivariant.
Its uniqueness is relative to the fixed pair $(\PP,\PP')$ and its
simultaneous transforms.
\end{proof}

\subsection{Homogeneous restriction}\label{subsec:homogeneous}

We first prove the transfer theorem when $\theta$ is invariant
under the normal intermediate subgroup $J$. The hypothesis includes
the full ordinary block isomorphism, with its common defect group and
normalizer condition, together with equality of the mixed comparison
functions for the same pair of associated representations.

\begin{theorem}[Transfer under homogeneous restriction]\label{thm:2-7}
Suppose that
\[
  (G,N,\theta)_{\HH}\geq_b(H,M,\varphi)_{\HH},
  \qquad M=N\cap H,
\]
and fix a pair of projective representations
\[
  \PP:G_\theta\longrightarrow\GL_{\theta(1)}(\Qab),\qquad
  \PP':H_\varphi\longrightarrow\GL_{\varphi(1)}(\Qab)
\]
affording this relation. Denote the corresponding intermediate-group
maps by $\tau_U$. Let
\[
  N\leq J\unlhd G,\qquad J\leq G_\theta,\qquad L=J\cap H.
\]
For $\psi\in\Irr(J\mid\theta)$, set
$\psi'=\tau_J(\psi)\in\Irr(L\mid\varphi)$. Then
\begin{equation}\label{eq:2-6}
  (G_{\psi^{\HH}},J,\psi)_{\HH}
  \geq_b(H_{(\psi')^{\HH}},L,\psi')_{\HH}.
\end{equation}
The blocks $\bl(\psi)$ and $\bl(\psi')$ have a common defect
group, and
\begin{equation}\label{eq:2-7}
  \htc(\psi)-\htc(\theta)=\htc(\psi')-\htc(\varphi).
\end{equation}
In particular, if $\theta$ and $\varphi$ have height zero, then
$\psi$ has height zero if and only if $\psi'$ does.
\end{theorem}

\begin{proof}
We first determine the stabilizers and then construct a pair of
projective representations affording \eqref{eq:2-6}. Since
$J\leq G_\theta$, Clifford theory and the fixed tensor correspondence
give
\[
  \psi_N=e\theta,\qquad \psi'_M=e\varphi,\qquad
  e=\frac{\psi(1)}{\theta(1)}=\frac{\psi'(1)}{\varphi(1)}.
\]
Let $a=(h,\sigma)\in H\times\HH$. If $\psi^a=\psi$, restriction
to $N$ gives $\theta^a=\theta$, and equality of the mixed
stabilizers of $\theta$ and $\varphi$ gives
$\varphi^a=\varphi$. Since $J\unlhd G$, covariance for the fixed
pair, as in Lemma~\ref{lem:2-3} and \cite[Lemma~1.9(a)]{NSV}, gives
\[
  \tau_J(\psi^a)=\tau_J(\psi)^a.
\]
Hence $(\psi')^a=\psi'$. Conversely, if $(\psi')^a=\psi'$, then
restriction to $M$ gives $\varphi^a=\varphi$, hence
$\theta^a=\theta$, and injectivity of $\tau_J$ gives $\psi^a=\psi$.
Thus
\begin{equation}\label{eq:2-8}
  (H\times\HH)_\psi=(H\times\HH)_{\psi'}.
\end{equation}
Write
\[
  X=G_{\psi^{\HH}},\qquad Y=H_{(\psi')^{\HH}},\qquad
  T=G_\psi,\qquad V=H_{\psi'}.
\]
Since $G=NH$, $N\leq J$, and inner automorphisms induced by $J$
fix $\psi$, equation~\eqref{eq:2-8} gives
\begin{equation}\label{eq:2-9}
  X=JY,\qquad J\cap Y=L,\qquad T=NV,\qquad V=T\cap H.
\end{equation}
The identity $\psi_N=e\theta$ also gives $T\leq G_\theta$, so all
subsequent evaluations of $\PP$ lie in its domain.

Let $\alpha$ be the factor set of $\PP$. Choose an irreducible
projective representation $\QQ$ of $J/N\cong L/M$, with factor
set $\alpha^{-1}|_{J\times J}$ after inflation, such that
\[
  \psi=\tr(\QQ\otimes\PP_J),\qquad
  \psi'=\tr(\QQ_L\otimes\PP'_L).
\]
By \cite[Theorem~1.1]{NSV}, its entries may be taken in a finite
cyclotomic field. Choose a projective representation $\mathcal D$
associated with $(T,J,\psi)$ such that
$\mathcal D_J=\QQ\otimes\PP_J$, and denote its factor set by
$\beta$. We may choose $\mathcal D$ with entries in a cyclotomic
field and with $\beta$ taking values in roots of unity. Indeed, start
with such an associated representation and conjugate it over a
sufficiently large cyclotomic field so that its restriction to $J$ is
the prescribed representation $\QQ\otimes\PP_J$. A conjugating
matrix exists because the restrictions afford the same irreducible
character.

For $t\in T$, the matrices $\mathcal D(t)$ and
$I_e\otimes\PP(t)$ implement the same conjugation on $N$.
Their quotient therefore centralizes $I_e\otimes\PP(N)$. Since
$\PP_N$ is irreducible, there is a unique matrix $\RR(t)$ such that
\begin{equation}\label{eq:2-10}
  \mathcal D(t)=\RR(t)\otimes\PP(t).
\end{equation}
Multiplication shows that $\RR$ has factor set $\beta\alpha^{-1}$.
It is constant on $N$-cosets and satisfies $\RR_J=\QQ$. Define
\[
  \mathcal D'(v)=\RR(v)\otimes\PP'(v)\qquad(v\in V).
\]
The factor set of $\PP'_V$ is $\alpha|_{V\times V}$, so that of
$\mathcal D'$ is $\beta|_{V\times V}$. Its restriction to $L$
affords $\psi'$, and the multiplication conditions for an associated
representation follow from those of $\RR$ and $\PP'$. Thus
$\mathcal D$ and $\mathcal D'$ are associated with the required
upper triples.

Restrict the given ordinary block isomorphism from $G_\theta,H_\varphi$
to $T,V$. This is permitted by $T=NV$, $V=T\cap H$, and
$T\leq G_\theta$. Since $J\unlhd T$ and $\psi$ is $T$-invariant,
the construction above is the tensor construction of
\cite[Proposition~3.9(b)]{NS}. It gives the prescribed
intermediate-group correspondences: if $J\leq U\leq T$ and
\[
  \rho=\tr(\mathcal S\otimes\mathcal D_U)\in\Irr(U\mid\psi),
\]
then its local image is
\[
  \tr(\mathcal S_{U\cap H}\otimes\mathcal D'_{U\cap H})
  =\tau_U(\rho),
\]
because $\mathcal S\otimes\RR_U$ is a projective representation
constant on $N$-cosets. The block-induction identities on all
intermediate groups are therefore inherited from the same original
correspondence. By \cite[Proposition~3.9(a),(b)]{NS}, the upper
blocks have a common defect group $Q_0$ with $\N_J(Q_0)\leq L$,
and the same result gives \eqref{eq:2-7}.

We also have $\C_X(J)\leq\C_G(N)\leq H$. Every element of
$\C_X(J)$ fixes $\psi$, so $\C_X(J)\leq V$. For $x\in\C_X(J)$,
both $\mathcal D(x)$ and $\PP(x)$ are scalar matrices.
Equation~\eqref{eq:2-10} implies that $\RR(x)$ is scalar as well.
Since $\PP(x)$ and $\PP'(x)$ have the same scalar, so do
$\mathcal D(x)$ and $\mathcal D'(x)$. Thus this pair affords the
required ordinary block isomorphism.

It remains to compare the mixed functions. Let
$a=(h,\sigma)\in(Y\times\HH)_\psi$. Equation~\eqref{eq:2-8}
implies that $a$ fixes $\psi'$; homogeneous restriction then shows
that it fixes $\theta$ and $\varphi$. Write
\[
  \begin{aligned}
    \PP^a&=\mu_aB_a^{-1}\PP B_a,&
    (\PP')^a&=\mu'_a(B'_a)^{-1}\PP'B'_a,\\
    \mathcal D^a&=\nu_aA_a^{-1}\mathcal D A_a,&
    \mu'_a&=\mu_a|_{H_\varphi}.
  \end{aligned}
\]
The group component of $a$ normalizes $T$ and $V$. The function
$\nu_a$ is constant on $J$-cosets and equals $1$ on $J$.
Restricting the comparison identity for $\mathcal D$ to $N$, and
using $\mathcal D_N=I_e\otimes\PP_N$ and $\mu_a|_N=1$, shows
that $A_a(I_e\otimes B_a^{-1})$ centralizes $I_e\otimes\PP(N)$.
Schur's lemma gives an invertible matrix $S_a$ such that
\[
  A_a=S_a\otimes B_a.
\]
Substituting this expression into the identity for $\mathcal D^a$
and cancelling the invertible second tensor factor gives
\begin{equation}\label{eq:2-11}
  \RR^a(t)=\frac{\nu_a(t)}{\mu_a(t)}S_a^{-1}\RR(t)S_a
  \qquad(t\in T).
\end{equation}
Consequently, for $v\in V$,
\[
  \begin{aligned}
    (\mathcal D')^a(v)
      &=\RR^a(v)\otimes(\PP')^a(v)\\
      &=\nu_a(v)(S_a\otimes B'_a)^{-1}
        \mathcal D'(v)(S_a\otimes B'_a).
  \end{aligned}
\]
Uniqueness of the normalized comparison function gives
$\nu'_a=\nu_a|_V$. Together with \eqref{eq:2-8}, \eqref{eq:2-9},
and the ordinary block isomorphism established above, this proves
\eqref{eq:2-6}.
\end{proof}

\subsection{Induced representations and mixed comparison functions}

We now construct the representations and intertwining matrices needed
for the nonhomogeneous Clifford step. Recall our convention for the
right action:
\[
  f^{(h,\gamma)}(x)=\gamma\bigl(f(hxh^{-1})\bigr).
\]
For matrix-valued functions, $\gamma$ acts entrywise, and
$(f^a)^b=f^{ab}$. For an associated projective representation, we write
the comparison identity as
\[
  \RR^a(x)=\kappa_a(x)A_a^{-1}\RR(x)A_a.
\]
Thus $A_a^{-1}$ plays the role of the intertwining matrix
$T_a$ in \eqref{eq:comparison}.

\subsubsection{A common transversal}

Assume that
\[
  (G,N,\theta)_{\HH}\geq_b(H,M,\varphi)_{\HH},
  \qquad G=NH,\qquad M=N\cap H.
\]
Let $N\leq J\unlhd G$, and set
\[
  L=J\cap H,\qquad T=J_\theta,\qquad U=T\cap H.
\]
Then $T\unlhd G$ and $U=L_\varphi$. Fix
\[
  \eta\in\Irr(T\mid\theta),\qquad
  \eta'=\tau_T(\eta),\qquad
  \psi=\eta^J,\qquad \psi'=(\eta')^L.
\]
Both induced characters are irreducible. The homogeneous transfer
theorem provides a fixed pair of associated projective representations
\[
  \RR:G_\eta\longrightarrow\GL_{\eta(1)}(F),\qquad
  \RR':H_{\eta'}\longrightarrow\GL_{\eta'(1)}(F),
\]
where $F$ is a sufficiently large finite cyclotomic field. Their factor
sets take values in roots of unity and agree under the natural
identification of the corresponding quotients. This pair affords the
fixed homogeneous correspondences on intermediate groups. Moreover,
\begin{equation}\label{eq:2-12}
  (H\times\HH)_\eta=(H\times\HH)_{\eta'}.
\end{equation}
For each $a$ in this mixed stabilizer, there are invertible matrices
$A_a$ and $A_a'$ such that
\begin{align}
  \RR^a&=\kappa_a A_a^{-1}\RR A_a,
  \label{eq:2-13}\\
  (\RR')^a&=
  \bigl(\kappa_a|_{H_{\eta'}}\bigr)(A_a')^{-1}\RR'A_a'.
  \label{eq:2-14}
\end{align}
The function $\kappa_a$ is constant on $T$-cosets and equals $1$ on $T$.
Write
\[
  K=G_\psi,\qquad K_1=G_\eta=G_{\psi,\theta},\qquad
  V=H_{\psi'},\qquad V_1=H_{\eta'}.
\]
Clifford theory and the homogeneous relation give
\begin{equation}\label{eq:2-15}
  \begin{gathered}
    K=JK_1,\qquad K_1\cap J=T,\qquad V=LV_1,\\
    V_1\cap L=U,\qquad K_1=TV_1.
  \end{gathered}
\end{equation}
The natural maps identify the quotients
\begin{equation}\label{eq:2-16}
  K_1/T\cong K/J\cong V/L\cong V_1/U.
\end{equation}
Let $\alpha_0$ denote the common normalized factor set of $\RR$ and
$\RR'$ on this quotient, and let $\alpha$ be its inflation to $K$.

Choose a right transversal
\[
  r_1=1,r_2,\ldots,r_m\in L
\]
for $U$ in $L$. Since $J=NL$ and $N\leq T$, these elements also form
right transversals for $T$ in $J$, for $K_1$ in $K$, and for $V_1$ in
$V$. Thus the right cosets have the form $Tr_i$, $K_1r_i$, and so on.

\subsubsection{Induced representations and their factor sets}

For $x\in K$, define the block matrix $\mathcal D(x)$, with blocks of
size $\eta(1)$, by
\begin{equation}\label{eq:2-17}
  \mathcal D(x)_{ij}=
  \begin{cases}
    \RR(r_ixr_j^{-1}),&r_ixr_j^{-1}\in K_1,\\
    0,&r_ixr_j^{-1}\notin K_1.
  \end{cases}
\end{equation}
For $x\in V$, use the same representatives to define
\begin{equation}\label{eq:2-18}
  \mathcal D'(x)_{ij}=
  \begin{cases}
    \RR'(r_ixr_j^{-1}),&r_ixr_j^{-1}\in V_1,\\
    0,&r_ixr_j^{-1}\notin V_1.
  \end{cases}
\end{equation}

\begin{lemma}\label{lem:2-8}
The projective representations $\mathcal D$ and $\mathcal D'$ are
associated with $(K,J,\psi)$ and $(V,L,\psi')$, respectively. Their
factor sets are $\alpha$ and $\alpha|_{V\times V}$.
\end{lemma}

\begin{proof}
For each $i$ and $x\in K$, there is a unique $j$ such that
$r_ixr_j^{-1}\in K_1$. Hence $\mathcal D(x)$ is an invertible block
permutation matrix. If $r_ixr_j^{-1}\in K_1$ and
$r_jyr_k^{-1}\in K_1$, then
\[
  \begin{aligned}
    \RR(r_ixr_j^{-1})\RR(r_jyr_k^{-1})
      &=\alpha_0\bigl((r_ixr_j^{-1})T,(r_jyr_k^{-1})T\bigr)
        \RR(r_ixyr_k^{-1})\\
      &=\alpha(x,y)\RR(r_ixyr_k^{-1}).
  \end{aligned}
\]
The last equality follows from $r_i,r_j,r_k\in J$ and the
identifications in \eqref{eq:2-16}. Block multiplication therefore gives
$\mathcal D(x)\mathcal D(y)=\alpha(x,y)\mathcal D(xy)$.
Since $\alpha$ is inflated from $K/J$, it equals $1$ whenever either
argument lies in $J$. The restriction $\mathcal D_J$ is the ordinary
representation induced from $\RR_T$, so it affords $\eta^J=\psi$.
These observations also verify the normalized multiplication
conditions for an associated projective representation.

If $x\in V$, then $r_ixr_j^{-1}\in V$, and this element belongs to
$K_1$ if and only if it belongs to $V_1$. Thus the local matrices have
the same permutation pattern. The same calculation gives the factor
set $\alpha|_{V\times V}$ for $\mathcal D'$, whose restriction to $L$
affords $\psi'$.
\end{proof}

\subsubsection{Adjusting a mixed stabilizer element}

\begin{lemma}\label{lem:2-9}
We have
\[
  (H\times\HH)_\psi=(H\times\HH)_{\psi'}.
\]
For every $a=(h,\gamma)$ in this group, there is an element $l\in L$
such that
\begin{equation}\label{eq:2-19}
  a_0=a(l^{-1},1)=(hl^{-1},\gamma)\in(H\times\HH)_\eta.
\end{equation}
The element $a_0$ also fixes $\theta$, $\varphi$, and $\eta'$.
Its group component normalizes $K$, $K_1$, $V$, $V_1$, $T$, $U$,
$J$, and $L$.
\end{lemma}

\begin{proof}
Suppose first that $\psi^a=\psi$. Restriction to $N$ and Clifford
theory show that $\theta^a$ is $J$-conjugate to $\theta$. Since $J=NL$
and inner automorphisms from $N$ fix $\theta$, we may choose $l\in L$
such that $\theta^a=\theta^l$. Then $a_0=a(l^{-1},1)$ fixes $\theta$
and, as $l\in J$, also fixes $\psi$. Induction gives the Clifford
bijection from $\Irr(T\mid\theta)$ to $\Irr(J\mid\theta)$, and $T$ is
normal in $G$. Consequently, $\eta^{a_0}$ and $\eta$ are both the
Clifford correspondent of $\psi$ over $\theta$, so $\eta^{a_0}=\eta$.
By \eqref{eq:2-12}, $a_0$ also fixes $\eta'$. Since induction commutes
with the group and Galois actions, $(\psi')^{a_0}=\psi'$. As $l\in L$,
we obtain $(\psi')^a=\psi'$.

Conversely, suppose that $(\psi')^a=\psi'$. Applying the same argument
to the Clifford constituents over $M$, choose $l\in L$ such that
$a_0=a(l^{-1},1)$ fixes $\varphi$. Equality of the mixed stabilizers
of $\theta$ and $\varphi$ implies that $a_0$ fixes $\theta$.
Uniqueness of the local Clifford correspondent gives
$(\eta')^{a_0}=\eta'$. Equation~\eqref{eq:2-12} then yields
$\eta^{a_0}=\eta$, and induction gives $\psi^a=\psi$.

Because the group and Galois actions commute, the group component of
a mixed element fixing a character normalizes its ordinary inertia
group. This proves the assertion for $K$, $K_1$, $V$, and $V_1$.
The assertions for $J$, $L$, $T$, and $U$ follow from normality and
the fact that $a_0\in H\times\HH$.
\end{proof}

\subsubsection{Intertwining matrices and comparison functions}

\begin{proposition}[Equality of mixed comparison functions]
\label{prop:2-10}
The fixed induced pair $\mathcal D,\mathcal D'$ has equal mixed
comparison functions on the common domain. More precisely, let
$a\in(H\times\HH)_\psi$, and choose $l$ and $a_0$ as in
\eqref{eq:2-19}. For $x\in K$, choose $t_x\in K_1$ with $t_xJ=xJ$.
The comparison function for $\mathcal D$ is given by
\begin{equation}\label{eq:2-20}
  \mu_a(x)=\kappa_{a_0}(t_x).
\end{equation}
This value is independent of both $t_x$ and the choice of $l$.
Moreover, the comparison function $\mu_a'$ for $\mathcal D'$ satisfies
\[
  \mu_a'(x)=\mu_a(x)\qquad(x\in V).
\]
\end{proposition}

\begin{proof}
First suppose that $a_0=(h_0,\gamma)$ fixes $\eta$. Since $h_0$
normalizes $T$ and $L$, there are a unique permutation $\pi$ and
unique elements $u_i\in U$ such that
\begin{equation}\label{eq:2-21}
  h_0^{-1}r_ih_0=u_ir_{\pi(i)}.
\end{equation}
Indeed, both $h_0^{-1}r_ih_0$ and the representatives lie in $L$,
while the left factor in the coset decomposition lies in $T$.
That factor therefore belongs to $T\cap L=U$.

Write $A=A_{a_0}$, and define the block permutation matrix $B_{a_0}$
by specifying its nonzero blocks:
\begin{equation}\label{eq:2-22}
  (B_{a_0})_{i,\pi(i)}=A^{-1}\RR(u_i).
\end{equation}
Let $\nu_{a_0}:K\longrightarrow F^\times$ be the inflation of
$\kappa_{a_0}$ through $K_1/T\cong K/J$. This function is constant
on $J$-cosets and equals $1$ on $J$. We claim that
\begin{equation}\label{eq:2-23}
  \mathcal D^{a_0}(x)
    =\nu_{a_0}(x)B_{a_0}\mathcal D(x)B_{a_0}^{-1}.
\end{equation}
To verify this identity block by block, use \eqref{eq:2-21} to write
\[
  r_ih_0xh_0^{-1}r_j^{-1}
    =h_0u_i\bigl(r_{\pi(i)}xr_{\pi(j)}^{-1}\bigr)u_j^{-1}h_0^{-1}.
\]
Since $h_0$ normalizes $K_1$ and $u_i,u_j\in T\leq K_1$, the
corresponding induced blocks vanish simultaneously. For a nonzero
block, set $t=r_{\pi(i)}xr_{\pi(j)}^{-1}\in K_1$.
Equation~\eqref{eq:2-13} gives
\[
  \begin{aligned}
    \mathcal D^{a_0}(x)_{ij}
      &=\RR(h_0u_itu_j^{-1}h_0^{-1})^\gamma\\
      &=\kappa_{a_0}(u_itu_j^{-1})A^{-1}\RR(u_itu_j^{-1})A\\
      &=\nu_{a_0}(x)A^{-1}\RR(u_i)
        \mathcal D(x)_{\pi(i),\pi(j)}\RR(u_j)^{-1}A.
  \end{aligned}
\]
The last equality uses that $\kappa_{a_0}$ is constant on $T$-cosets
and that the factor set equals $1$ whenever either argument lies in
$T$. The resulting expression is precisely the $(i,j)$ block on the
right-hand side of \eqref{eq:2-23}.

For the local representation, use the same $\pi$ and $u_i$, and define
\[
  (B_{a_0}')_{i,\pi(i)}=(A_{a_0}')^{-1}\RR'(u_i).
\]
The same block calculation, now using \eqref{eq:2-14}, gives
\begin{equation}\label{eq:2-24}
  (\mathcal D')^{a_0}(x)
    =\nu_{a_0}(x)B_{a_0}'\mathcal D'(x)(B_{a_0}')^{-1}
    \qquad(x\in V).
\end{equation}

Now let $a=a_0(l,1)$ be arbitrary. Since $l\in J$, the normalized
multiplication conditions give
\begin{equation}\label{eq:2-25}
  \mathcal D^{(l,1)}(x)=\mathcal D(lxl^{-1})
    =\mathcal D(l)\mathcal D(x)\mathcal D(l)^{-1}.
\end{equation}
Thus the comparison function for $(l,1)$ is identically $1$.
To combine this identity with \eqref{eq:2-23}, we use the composition
rule for the right action. If
$\mathcal D^a=\mu_aB_a\mathcal D B_a^{-1}$ and
$\mathcal D^b=\mu_bB_b\mathcal D B_b^{-1}$, then
\begin{equation}\label{eq:2-26}
  \mathcal D^{ab}
    =\mu_a^b\mu_b(B_a^{\gamma_b}B_b)
      \mathcal D(B_a^{\gamma_b}B_b)^{-1},
\end{equation}
where $\gamma_b$ denotes the Galois component of $b$. For $b=(l,1)$,
we have $\gamma_b=1$. Moreover,
$\nu_{a_0}(lxl^{-1})=\nu_{a_0}(x)$ because $\nu_{a_0}$ is inflated
from $K/J$ and $l\in J$. Substituting \eqref{eq:2-23} and
\eqref{eq:2-25} into \eqref{eq:2-26} yields
\[
  \mathcal D^a(x)=\nu_{a_0}(x)
  \bigl(B_{a_0}\mathcal D(l)\bigr)\mathcal D(x)
  \bigl(B_{a_0}\mathcal D(l)\bigr)^{-1}.
\]
Similarly, on the local side,
\[
  (\mathcal D')^a(x)=\nu_{a_0}(x)
  \bigl(B_{a_0}'\mathcal D'(l)\bigr)\mathcal D'(x)
  \bigl(B_{a_0}'\mathcal D'(l)\bigr)^{-1}
  \qquad(x\in V).
\]
Here $\mathcal D'(l)$ is defined because $l\in L$.
The scalar functions in these identities are those in
\eqref{eq:2-20}, so they agree on $V$.

It remains to verify independence of the choices. Replacing $t_x$
changes only the representative of its $T$-coset in $K_1$, and hence
does not change $\kappa_{a_0}(t_x)$. Two choices of $l$ give two
intertwining matrices and their comparison functions. Both functions
equal $1$ on $J$, so the two matrices intertwine the
same irreducible representations $\mathcal D_J$ and
$(\mathcal D_J)^a$. By Schur's lemma, they differ by a scalar.
Substitution into the identities on $K$ then shows that the two
comparison functions agree pointwise. The same argument applies
locally.
\end{proof}

\subsubsection{Correspondences on intermediate groups}

\begin{proposition}\label{prop:2-11}
Let $J\leq W\leq K$, and set
\[
  W_1=W\cap K_1,\qquad W'=W\cap V,\qquad W_1'=W_1\cap V_1.
\]
For $\chi\in\Irr(W\mid\psi)$, let
$\chi_1\in\Irr(W_1\mid\eta)$ be its Clifford correspondent.
The correspondence afforded by $\mathcal D,\mathcal D'$ is
\begin{equation}\label{eq:2-27}
  \Sigma_W(\chi)=\bigl(\tau_{W_1}(\chi_1)\bigr)^{W'}.
\end{equation}
Here $\tau_{W_1}$ is the homogeneous correspondence afforded by
$\RR,\RR'$, and hence the appropriate restriction of the original
correspondence afforded by $(\PP,\PP')$.
\end{proposition}

\begin{proof}
Equation~\eqref{eq:2-15} gives $W=JW_1$ and $W'=LW_1'$.
Thus the same elements $r_i$ form right transversals for $W_1$ in
$W$ and for $W_1'$ in $W'$. Since $W_1=W_\theta=W_\eta$,
Clifford theory gives the unique correspondent $\chi_1$ satisfying
$\chi_1^W=\chi$. Choose a projective representation $\QQ$ of the
common quotient $W_1/T\cong W/J$, with the restriction of
$\alpha_0^{-1}$ as its factor set, such that
\[
  \chi_1=\tr(\QQ\otimes\RR_{W_1}),\qquad
  \tau_{W_1}(\chi_1)=\tr(\QQ\otimes\RR'_{W_1'}).
\]
Let $\widetilde{\QQ}$ be its inflation from $W/J$ to $W$.
In each nonzero block of the induced matrices, we have
\[
  \QQ(r_ixr_j^{-1})=\widetilde{\QQ}(x),
\]
because $r_i,r_j\in J$. After a fixed permutation of the tensor-factor
basis, the matrices obtained by inducing $\QQ\otimes\RR_{W_1}$ are
therefore $\widetilde{\QQ}\otimes\mathcal D_W$.
The same calculation on the local side gives
$\widetilde{\QQ}_{W'}\otimes\mathcal D'_{W'}$.
Consequently, the triple correspondence defined using the common
factor $\widetilde{\QQ}$ is exactly \eqref{eq:2-27}.
\end{proof}

\subsubsection{Central scalars}

\begin{lemma}\label{lem:2-12}
Let $X=G_{\psi^{\HH}}$. If $c\in\C_X(J)$, then $\mathcal D(c)$
and $\mathcal D'(c)$ are scalar matrices with the same scalar.
\end{lemma}

\begin{proof}
Since $N\leq J$, we have $\C_X(J)\leq\C_G(N)\leq H$.
The element $c$ centralizes $J$ and hence fixes $\psi$; it also
centralizes $N$ and hence fixes $\theta$. Thus
$c\in K_1\cap V=V_1$. Each $r_i$ belongs to $J$ and therefore
commutes with $c$. It follows that
\[
  \begin{aligned}
    \mathcal D(c)&=\operatorname{diag}\bigl(\RR(c),\ldots,\RR(c)\bigr),\\
    \mathcal D'(c)&=\operatorname{diag}\bigl(\RR'(c),\ldots,\RR'(c)\bigr).
  \end{aligned}
\]
Since $c$ centralizes $T$, the matrix $\RR(c)$ is scalar.
The central-scalar condition for the homogeneous relation implies
that $\RR'(c)$ has the same scalar.
\end{proof}

The induced pair therefore has equal factor sets and mixed comparison
functions on the common domains. We have also identified its
correspondences on intermediate groups and verified the central-scalar
condition. The block arguments below establish the common defect group,
its normalizer condition, and the assertions concerning heights.

\subsection{Blocks, defect groups, and relative heights}\label{subsec:2-4}

Fix a pair of associated projective representations affording
\[
  (G,N,\theta)_{\HH}\geq_b(H,M,\varphi)_{\HH},
  \qquad G=NH,\qquad M=N\cap H,
\]
and let $\tau$ denote the resulting family of character correspondences.
Write $b=\bl(\theta)$ and $b'=\bl(\varphi)$, and let $D\leq M$ be
the common defect group specified by the given relation. Thus
\begin{equation}\label{eq:2-28}
  \N_N(D)\leq M,\qquad (b')^N=b.
\end{equation}
The second equality is the block condition at $N$. For $N\leq V\leq G$,
write $V'=V\cap H$. Throughout this subsection, $\theta$ and $\varphi$
may have arbitrary heights, and $V$ need not be contained in $G_\theta$.

\begin{lemma}[Normalizer of the defect group $D$]\label{lem:2-13}
We have
\[
  H=M\N_G(D),\qquad \N_G(D)\leq H.
\]
Consequently, for every $N\leq V\leq G$,
\begin{equation}\label{eq:2-29}
  V=NV',\qquad V'=M\N_V(D),\qquad \N_V(D)\leq V'.
\end{equation}
\end{lemma}

\begin{proof}
By Lemma~\ref{lem:2-1}, the specified subgroup $D$ is a defect group
of every Galois conjugate of $b'$ under $\HH$. For $h\in H$,
semi-invariance of the local $\HH$-triple gives $\gamma\in\HH$ such
that $\varphi^h=\varphi^\gamma$, and hence $(b')^h=(b')^\gamma$.
Thus $D^h$ and $D$ are defect groups of the same block and are conjugate
in $M$. It follows that $H=M\N_H(D)$. Since $G=NH$ and $M\leq N$,
we also have $G=N\N_H(D)$. If $g\in\N_G(D)$, write $g=nh$ with
$n\in N$ and $h\in\N_H(D)$. Then $n\in\N_N(D)\leq M$, so
$g\in H$. This proves the first two assertions.

Since $N\leq V$, the factorization $G=NH$ gives
$V=N(V\cap H)=NV'$. Intersecting $H=M\N_G(D)$ with $V$ and using
$M\leq N\leq V$ gives the remaining assertions in
\eqref{eq:2-29}. This argument uses semi-invariance of the local
$\HH$-triple; it does not require $b'$ to be fixed by every element
of $\HH$.
\end{proof}

\begin{lemma}[Stabilizers of $b$ and $b'$]\label{lem:2-14}
We have $H_b=H_{b'}$. For $N\leq V\leq G$, set $W=V_b$ and
$W'=(V')_{b'}$. Then
\begin{equation}\label{eq:2-30}
  W=NW',\qquad W\cap V'=W',\qquad N\cap W'=M.
\end{equation}
\end{lemma}

\begin{proof}
Equation~\eqref{eq:2-28} gives $\N_M(D)=\N_N(D)$. Let $c_0$ be
the Brauer correspondent of $b'$ in $\N_M(D)$. Transitivity of
Brauer correspondence and the equality $(b')^N=b$ yield
\[
  c_0^M=b',\qquad c_0^N=b.
\]
Thus $c_0$ is the Brauer correspondent of both $b$ and $b'$ in the
common subgroup $\N_N(D)=\N_M(D)$.

Every $h\in\N_G(D)$ normalizes $N$, $M$, and $D$. Uniqueness in
Brauer's first main theorem, applied to both blocks, gives
\[
  b^h=b\quad\Longleftrightarrow\quad c_0^h=c_0
  \quad\Longleftrightarrow\quad (b')^h=b'.
\]
For an arbitrary $h\in H$, Lemma~\ref{lem:2-13} gives a factorization
$h=mh_0$ with $m\in M$ and $h_0\in\N_G(D)$. Conjugation by $m$
fixes both blocks, so the same equivalence holds for $h$. Therefore
$H_b=H_{b'}$.

Finally, $V=NV'$ and $N$ fixes $b$, so
$V_b=N(V'\cap H_b)=NW'$. The other two equalities in
\eqref{eq:2-30} follow by intersection.
\end{proof}

The next lemma formulates the argument of
\cite[Lemma~3.5(b)--(d)]{NS} in terms of blocks. It requires only
that the covered block be invariant; an irreducible character in that
block need not be invariant under the whole upper group.

\begin{lemma}[A common Harris--Kn\"orr correspondence]\label{lem:2-15}
With $W$ and $W'$ as in Lemma~\ref{lem:2-14}, block induction gives
a bijection
\[
  \begin{aligned}
    \{C'\in\Bl(W')\mid C'\text{ covers }b'\}
      &\longrightarrow \{C\in\Bl(W)\mid C\text{ covers }b\},\\
    C'&\longmapsto (C')^W.
  \end{aligned}
\]
For each corresponding pair $C'$ and $C=(C')^W$, there is a common
defect group $Q\leq W'$ such that
\begin{equation}\label{eq:2-31}
  Q\cap N=Q\cap M=D.
\end{equation}
\end{lemma}

\begin{proof}
Set $R_D=\N_W(D)$. By Lemma~\ref{lem:2-13}, we have $R_D\leq W'$
and $R_D=\N_{W'}(D)$. Since $b$ is $W$-invariant, $b'$ is
$W'$-invariant, and both blocks have defect group $D$, the Frattini
argument gives
\[
  W=NR_D,\qquad W'=MR_D.
\]
The proof of Lemma~\ref{lem:2-14} identifies a common local block
$c_0\in\Bl(\N_N(D))$. This block is $R_D$-invariant: the group
$R_D$ normalizes $N$ and $D$ and fixes $b$, so uniqueness of the
Brauer correspondent implies that it also fixes $c_0$.

Apply the Harris--Kn\"orr correspondence first to $N\unlhd W$ and
its invariant block $b$, and then to $M\unlhd W'$ and its invariant
block $b'$. In both applications, the local groups are $R_D$ and
$\N_N(D)=\N_M(D)$, and $c_0$ is the common invariant block of the
latter. Thus each block $C'$ covering $b'$ corresponds to a unique
block $c\in\Bl(R_D)$ covering $c_0$ such that $c^{W'}=C'$. The same
block $c$ corresponds to the unique block $C=c^W$ covering $b$.
Both block inductions are defined by the Harris--Kn\"orr theorem.
Since $R_D\leq W'\leq W$, transitivity gives $C=(C')^W$. The two
Harris--Kn\"orr bijections also give uniqueness of the inverse
correspondence.

These correspondences preserve defect groups. Choose a defect group
$Q\leq R_D$ of $c$. Then $Q$ is a defect group of both $C$ and
$C'$. Since $c$ covers $c_0$ and $\N_N(D)\unlhd R_D$, the
normal-subgroup intersection theorem gives $Q\cap\N_N(D)=D$.
Here $D\unlhd R_D$, so the equality holds for the specified subgroup
$D$. Since $Q\leq R_D$, every element of $Q$ normalizes $D$. Hence
\[
  Q\cap N=Q\cap\N_N(D)=D.
\]
Since $D\leq M\leq N$, we also have $Q\cap M=D$.
\end{proof}

\begin{proposition}[Blocks of the fixed Clifford correspondence]\label{prop:2-16}
Let $N\leq V\leq G$, write $V'=V\cap H$, and set
\[
  V_0=V_\theta,\qquad V_0'=V_0\cap H=(V')_\varphi.
\]
For $\chi_0\in\Irr(V_0\mid\theta)$, define
\[
  \chi_0'=\tau_{V_0}(\chi_0),\qquad
  \chi=\chi_0^V,\qquad \chi'=(\chi_0')^{V'}.
\]
Then $\chi$ and $\chi'$ are irreducible, and their blocks have a
common defect group $Q\leq V'$ satisfying
\begin{align}
  Q\cap N&=Q\cap M=D,\qquad \N_V(Q)\leq V',\label{eq:2-32}\\
  \bl(\chi')^V&=\bl(\chi),\label{eq:2-33}\\
  \htc(\chi)-\htc(\theta)&=\htc(\chi')-\htc(\varphi).
    \label{eq:2-34}
\end{align}
\end{proposition}

\begin{proof}
Equality of the ordinary stabilizers gives $H_\theta=H_\varphi$,
and hence $V_0'=(V')_\varphi$. Clifford correspondence for
$N\unlhd V$ and $M\unlhd V'$ shows that $\chi$ and $\chi'$ are
irreducible.

Set $W=V_b$ and $W'=(V')_{b'}$. Lemma~\ref{lem:2-14} gives
$W=NW'$, and we have $V_0\leq W$ and $V_0'\leq W'$. First induce
to these block stabilizers, setting
\[
  \chi_b=\chi_0^W,\qquad \chi'_{b'}=(\chi_0')^{W'},\qquad
  C=\bl(\chi_b),\qquad C'=\bl(\chi'_{b'}).
\]
These characters are irreducible because $W_\theta=V_0$ and
$(W')_\varphi=V_0'$. Write $d=\bl(\chi_0)$ and
$d'=\bl(\chi_0')$. Since $\chi_0'$ is defined by the fixed map
$\tau_{V_0}$, the given block relation yields
\[
  (d')^{V_0}=d.
\]
Since the induced characters are irreducible, the standard
block-induction formula gives $C=d^W$ and $C'=(d')^{W'}$; see
\cite[proof of Theorem~3.14, pp.~712--713]{NS}. Since $C'$ covers
$b'$, Lemma~\ref{lem:2-15} also ensures that $(C')^W$ is defined.
Applying transitivity along $V_0'\leq W'\leq W$ and
$V_0'\leq V_0\leq W$, we obtain
\begin{equation}\label{eq:2-35}
  (C')^W=\bigl((d')^{W'}\bigr)^W=(d')^W
    =\bigl((d')^{V_0}\bigr)^W=d^W=C.
\end{equation}
Thus $C$ and $C'$ are the corresponding blocks of
Lemma~\ref{lem:2-15}. Choose the common defect group $Q$ supplied
by that lemma, so that $Q\cap N=Q\cap M=D$.

Now apply Fong--Reynolds correspondence to $b$ for $N\unlhd V$ and
to $b'$ for $M\unlhd V'$. The block inertia groups are $W$ and
$W'$, respectively, and
\[
  \chi_b^V=\chi,\qquad (\chi'_{b'})^{V'}=\chi'.
\]
Consequently, $\bl(\chi)=C^V$ and $\bl(\chi')=(C')^{V'}$ are
the corresponding Fong--Reynolds blocks. This correspondence preserves
defect groups, so the same subgroup $Q$ is a defect group of both
$\bl(\chi)$ and $\bl(\chi')$.

Since $N\unlhd V$, every element of $\N_V(Q)$ also normalizes
$Q\cap N=D$. Lemma~\ref{lem:2-13} therefore gives
\[
  \N_V(Q)\leq\N_V(D)\leq V'.
\]
Brauer's first main theorem now ensures that the block $\bl(\chi')$,
whose defect group is $Q$, induces to $V$. Using \eqref{eq:2-35}
and transitivity once more gives
\[
  \bl(\chi')^V=\bigl((C')^{V'}\bigr)^V=(C')^V
    =\bigl((C')^W\bigr)^V=C^V=\bl(\chi).
\]
This proves \eqref{eq:2-32} and \eqref{eq:2-33}.

Finally, $V=NV'$ and $V_0=NV_0'$ imply
$[V:V_0]=[V':V_0']$. The fixed triple correspondence preserves
relative degrees, so
\[
  \frac{\chi(1)}{\chi'(1)}
    =\frac{[V:V_0]\chi_0(1)}{[V':V_0']\chi_0'(1)}
    =\frac{\theta(1)}{\varphi(1)}.
\]
The blocks of $\chi$ and $\chi'$ have common defect group $Q$,
the blocks $b$ and $b'$ have common defect group $D$, and
$[V:V']=[N:M]$. The definition of height therefore gives
\[
  \begin{aligned}
    \htc(\chi)-\htc(\chi')
      &=v_p\!\left(\frac{\chi(1)}{\chi'(1)}\right)
        -v_p\!\left(\frac{|V|}{|V'|}\right)\\
      &=v_p\!\left(\frac{\theta(1)}{\varphi(1)}\right)
        -v_p\!\left(\frac{|N|}{|M|}\right)\\
      &=\htc(\theta)-\htc(\varphi),
  \end{aligned}
\]
which is \eqref{eq:2-34}.
\end{proof}

\begin{corollary}[Height changes in a normal intermediate group]\label{cor:2-17}
Let $N\leq J\unlhd G$, and set $L=J\cap H$, $T=J_\theta$, and
$U=T\cap H$. For $\eta\in\Irr(T\mid\theta)$, define
\[
  \eta'=\tau_T(\eta),\qquad \psi=\eta^J,\qquad \psi'=(\eta')^L.
\]
Then $\bl(\psi)$ and $\bl(\psi')$ have a common defect group $Q$
with $\N_J(Q)\leq L$, and
\begin{equation}\label{eq:2-36}
  \htc(\psi)-\htc(\eta)=\htc(\psi')-\htc(\eta').
\end{equation}
\end{corollary}

\begin{proof}
Apply Proposition~\ref{prop:2-16} first with $V=J$. It gives the
common defect group, the inclusion $\N_J(Q)\leq L$, and the equality
\[
  \htc(\psi)-\htc(\psi')=\htc(\theta)-\htc(\varphi).
\]
Apply the same proposition with $V=T$. Since $T_\theta=T$, no
induction is needed, and we obtain
\[
  \htc(\eta)-\htc(\eta')=\htc(\theta)-\htc(\varphi).
\]
Subtracting proves \eqref{eq:2-36}. Thus, in this setting, the given
block relation implies the height-change and defect-normalizer
hypotheses of \cite[Theorem~3.14]{NS}.
\end{proof}

\begin{corollary}[Block conditions for every intermediate group]\label{cor:2-18}
With the notation of Corollary~\ref{cor:2-17}, consider the pair
associated with $\psi$ and $\psi'$ obtained by projective induction
from the homogeneous upper pair. For every $J\leq V\leq G_\psi$,
the tensor correspondence $\Omega_V$ afforded by this pair satisfies
\[
  \bl\bigl(\Omega_V(\chi)\bigr)^V=\bl(\chi)
  \qquad\bigl(\chi\in\Irr(V\mid\psi)\bigr).
\]
\end{corollary}

\begin{proof}
Set $G_1=G_\eta=G_{\psi,\theta}$ and
$H_1=H_{\eta'}=H_{\psi',\varphi}$. The group and Clifford
calculations above give $G_\psi=JG_1$, $G_1\cap J=T$, and the
corresponding local equalities. Hence, for $J\leq V\leq G_\psi$,
\[
  V_\theta=V\cap G_1.
\]
For $\chi\in\Irr(V\mid\psi)$, let
$\chi_0\in\Irr(V_\theta\mid\eta)$ be its Clifford correspondent,
so that $\chi_0^V=\chi$. The correspondence determined by the fixed
representations is
\begin{equation}\label{eq:2-37}
  \Omega_V(\chi)=\bigl(\tau_{V_\theta}(\chi_0)\bigr)^{V\cap H}.
\end{equation}
To verify this identity for the specified pair, recall that the
homogeneous construction on $\Irr(V_\theta\mid\eta)$ gives the
restriction of the original correspondence $\tau_{V_\theta}$.
Let $\RR$ be the homogeneous upper representation associated with
$\eta$, let $\mathcal D$ be its induced representation, and let $\QQ$
be a projective representation inflated from $V/J$. Then
\[
  \QQ\otimes\mathcal D_V\cong
  \operatorname{Ind}_{V_\theta}^{V}
    \bigl(\QQ|_{V_\theta}\otimes\RR_{V_\theta}\bigr).
\]
This identity follows by comparing the matrix blocks indexed by the
chosen representatives of $J/T$. The representation $\QQ$ takes
the identity matrix on these representatives, so moving its tensor
factor through the induced matrices introduces no scalar. Repeating
this argument for the primed representations on $V\cap H$, with the
same representatives viewed in $L/U$, gives the corresponding
tensor-induction identity.
Together with compatibility of the homogeneous pair with $\tau$,
this proves \eqref{eq:2-37}.

Apply Proposition~\ref{prop:2-16} to $V$ and $\chi_0$. Its block
equality, together with \eqref{eq:2-37}, proves the assertion.
Thus every intermediate-group block condition is satisfied by the
same induced pair whose mixed comparison functions were computed
in Subsection~2.3.
\end{proof}

\subsection{The Clifford transfer theorem}\label{subsec:2-5}

\begin{theorem}[Clifford correspondence]\label{thm:2-19}
Suppose that
\[
  (G,N,\theta)_{\HH}\geq_b(H,M,\varphi)_{\HH},
  \qquad N\leq J\unlhd G,
\]
and fix the representations, the family of correspondences, and the
common defect group $D$ in Definition~\ref{def:block}. Set
\[
  L=J\cap H,\qquad T=J_\theta,\qquad U=T\cap H=L_\varphi.
\]
For $\psi\in\Irr(J\mid\theta)$, let
$\eta\in\Irr(T\mid\theta)$ be its Clifford correspondent, and define
\[
  \eta'=\tau_T(\eta),\qquad \psi'=(\eta')^L.
\]
Then the following statements hold.
\begin{enumerate}[label=\textup{(\alph*)},ref=\textup{(\alph*)}]
\item The map $\psi\mapsto\psi'$ is a bijection from
  $\Irr(J\mid\theta)$ onto $\Irr(L\mid\varphi)$. Together with
  the maps obtained by simultaneous Galois transforms of the fixed
  pair of representations, it determines a unique
  $H\times\HH$-equivariant bijection
  \[
    \Omega_J\colon\Irr(J\mid\theta^{\HH})
      \longrightarrow\Irr(L\mid\varphi^{\HH}).
  \]
\item If $X=G_{\psi^{\HH}}$ and $Y=H_{(\psi')^{\HH}}$, then
  \begin{equation}\label{eq:2-38}
    (X,J,\psi)_{\HH}\geq_b(Y,L,\psi')_{\HH}.
  \end{equation}
  In particular,
  \[
    (H\times\HH)_\psi=(H\times\HH)_{\psi'},\qquad
    X=JY,\qquad J\cap Y=L.
  \]
  The relation in \eqref{eq:2-38} is afforded by the pair of induced
  projective representations $\mathcal D$ and $\mathcal D'$
  constructed in the preceding subsections.
\item For every $J\leq V\leq G_\psi$ and
  $\xi\in\Irr(V\mid\psi)$, the correspondence afforded by this
  pair satisfies
  \begin{equation}\label{eq:2-39}
    \widehat\tau_V(\xi)
      =\bigl(\tau_{V_\theta}(\xi_\theta)\bigr)^{V\cap H},
    \qquad
    \bl\bigl(\widehat\tau_V(\xi)\bigr)^V=\bl(\xi),
  \end{equation}
  where $\xi_\theta$ is the Clifford correspondent of $\xi$ over
  $\theta$.
\item The blocks of $\psi$ and $\psi'$ have a common defect group
  $Q$ such that
  \begin{equation}\label{eq:2-40}
    Q\cap N=Q\cap M=D,\qquad \N_J(Q)\leq L.
  \end{equation}
  Moreover,
  \begin{equation}\label{eq:2-41}
    \frac{\psi(1)}{\theta(1)}=\frac{\psi'(1)}{\varphi(1)},
    \qquad
    \htc(\psi)-\htc(\theta)=\htc(\psi')-\htc(\varphi).
  \end{equation}
\end{enumerate}
\end{theorem}

The theorem allows $J\nleq G_\theta$ and arbitrary heights for
$\theta$ and $\varphi$. If their heights are equal,
\eqref{eq:2-41} shows that the upper correspondence preserves
heights. The common defect group and its normalizer condition follow
from the given block relation.

\begin{proof}[Proof of Theorem~\ref{thm:2-19}]
Proposition~\ref{prop:2-6} gives the fiber bijection and its
$H\times\HH$-equivariant extension. By Lemma~\ref{lem:2-2}, we
have $T\unlhd G$ and $T\leq G_\theta$, so
Theorem~\ref{thm:2-7} applies to $T$. It supplies a fixed pair
$\RR,\RR'$ associated with $\eta,\eta'$ whose mixed comparison
functions agree and whose intermediate-group correspondences are
restrictions of the original family $\tau$.

Induce this pair using the common transversal for $U$ in $L$ chosen
in Subsection~2.3. This yields associated projective representations
$\mathcal D$ and $\mathcal D'$ for $\psi$ and $\psi'$.
Lemma~\ref{lem:2-8} identifies their common factor set.
Proposition~\ref{prop:2-10} proves equality of their comparison
functions for every mixed stabilizer element, including those that
permute the irreducible constituents of $\psi_N$ and $\psi'_M$.
Lemma~\ref{lem:2-12} gives equality of the central scalars. The
required group factorizations, intersections, and centralizer
containment follow from Lemma~\ref{lem:2-5}.

Proposition~\ref{prop:2-11} gives the first equality in
\eqref{eq:2-39}. Applying Proposition~\ref{prop:2-16} to these
same correspondences gives the block equality in that formula.
The case $V=J$ also supplies the common defect group and normalizer
condition in \eqref{eq:2-40}. Thus the same induced pair satisfies
every condition of Definition~\ref{def:block}, proving
\eqref{eq:2-38}.

Finally, the degree formula for Clifford induction and the equality
$[J:T]=[L:U]$ give equality of relative degrees. The height
calculation in Subsection~\ref{subsec:2-4} proves the remaining
equality in \eqref{eq:2-41}.
\end{proof}

\begin{corollary}[Height-zero fibers with a specified defect group]\label{cor:2-20}
In the setting of Theorem~\ref{thm:2-19}, suppose that $\theta$
and $\varphi$ have equal heights, and let $Q\leq L$ be a
$p$-subgroup with $Q\cap N=D$. Let
$\Irr_0(J\mid Q,\theta^{\HH})$ denote the height-zero characters
lying in blocks with defect group $Q$ and covering a character in
$\theta^{\HH}$. Define $\Irr_0(L\mid Q,\varphi^{\HH})$
analogously. Then the same map restricts to an
$\N_H(Q)\times\HH$-equivariant bijection
\[
  \Omega_J\colon\Irr_0(J\mid Q,\theta^{\HH})
    \longrightarrow\Irr_0(L\mid Q,\varphi^{\HH}).
\]
Every corresponding pair satisfies the block $\HH$-triple relation
in Theorem~\ref{thm:2-19}.
\end{corollary}

\begin{proof}
First fix a corresponding pair $\psi,\psi'$. Suppose that $Q$ is
a defect group of $\bl(\psi)$ and $Q\cap N=D$.
Theorem~\ref{thm:2-19} gives a common defect group $Q_0$ of the
two blocks with $Q_0\cap N=D$. Choose $j\in J$ such that
$Q_0^j=Q$. Since $N\unlhd J$,
\[
  D^j=(Q_0\cap N)^j=Q\cap N=D.
\]
Hence $j\in\N_J(D)\leq L$. Conjugation by $j$ fixes
$\bl(\psi')$, so $Q$ is also a defect group of this local block.
Conversely, if $Q$ is a defect group of $\bl(\psi')$, conjugating
$Q_0$ to $Q$ inside $L$ shows that $Q$ is also a defect group of
$\bl(\psi)$.

The same argument applies to every Galois conjugate of the given
relation, since Lemma~\ref{lem:2-1} preserves the specified common
defect group $D$. The height equality and bijectivity in
Theorem~\ref{thm:2-19} therefore give the asserted bijection on
height-zero characters. The group $\N_H(Q)$ fixes $Q$, while
$\HH$ preserves its occurrence as a defect group. Equivariance
thus follows by restricting the $H\times\HH$-equivariance of
$\Omega_J$.
\end{proof}

The argument relies on the normality of $J$ in $G$, the factorization
$G=NH$, and semi-invariance of the given $\HH$-triple. It does not
assert that arbitrary irreducible character inductions preserve a
block $\HH$-triple relation. The hypotheses of
\cite[Theorem~3.14]{NS} cannot in general be omitted. In the present
setting, the required block-theoretic hypotheses follow from
Proposition~\ref{prop:2-16} and Corollary~\ref{cor:2-17}, while
Lemma~\ref{lem:2-9} supplies the adjustment of mixed stabilizer
elements needed to compare the induced representations.

Section~3 assembles these correspondences over families of character
fibers and treats induction when the two Clifford indices differ.

\section{Induction and gluing of correspondences}

We establish two gluing results for block $\HH$-triple relations.
The first assembles correspondences on a normal subgroup into a
correspondence on a larger normal subgroup and provides a Galois
analogue of \cite[Proposition~4.7]{NS}. The second assembles
correspondences on inertia groups, as in
\cite[Proposition~7.4]{NS}. In the second construction, the two
Clifford indices need not be equal. We therefore begin with an
induction lemma that uses separate sets of coset representatives.

Throughout this section, we use the right action
\[
  \PP^{h\sigma}(x)=\PP(hxh^{-1})^\sigma.
\]
Comparison functions are normalized to equal $1$ on the relevant
normal subgroup and are constant on its cosets. Whenever a block
isomorphism is afforded by a specified pair of projective
representations, we retain that pair throughout the construction.

\subsection{Compatible choices on Galois orbits}

In each application, we first choose representatives of the relevant
group--Galois orbits on characters of the normal subgroup and fix
the required data for these representatives. We then transport both the maps on character fibers and the
pairs of representations affording them. Equivariance under the
full mixed stabilizer ensures that different transporting elements
give the same correspondence. We also check compatibility with
Clifford conjugacy, so that the resulting map is independent of the
chosen irreducible constituent on the normal subgroup.

Section~2 treats the passage from a fixed pair of characters to
characters of larger groups. Here the initial data range over a
family of character fibers. The gluing construction in
Subsection~3.4 may involve unequal Clifford indices; the following
lemma supplies the required induced representations.

\subsection{Induction with unequal Clifford indices}

\begin{lemma}\label{lem:3-1}
Let $N\unlhd A$, let $B\leq A$, and set $M=N\cap B$. Suppose that
\[
  A=NB,\qquad \psi\in\Irr(N),\qquad \psi'\in\Irr(M),
\]
that $\psi^a\in\psi^{\HH}$ for every $a\in A$ and
$(\psi')^b\in(\psi')^{\HH}$ for every $b\in B$, and that
\begin{equation}\label{eq:3-1}
  (B\times\HH)_\psi=(B\times\HH)_{\psi'}=:\Gamma.
\end{equation}
Set $G=A_\psi$ and $H=B_{\psi'}$. Let $G_1\leq G$, and write
\[
  N_1=N\cap G_1,\qquad H_1=H\cap G_1,\qquad M_1=M\cap G_1.
\]
Suppose that $G=NG_1$, $H=MH_1$, and that fixed projective
representations
\[
  \RR:G_1\longrightarrow\GL_{\eta(1)}(F),\qquad
  \RR':H_1\longrightarrow\GL_{\eta'(1)}(F)
\]
afford a block isomorphism
\begin{equation}\label{eq:3-2}
  (G_1,N_1,\eta)\sim_b(H_1,M_1,\eta').
\end{equation}
Here $F$ is a sufficiently large finite cyclotomic field,
$\eta$ is $G_1$-invariant, $\eta'$ is $H_1$-invariant, and
$\eta^N=\psi$, $(\eta')^M=\psi'$. Assume that the factor sets of
$\RR$ and $\RR'$ are inflated from the same root-of-unity-valued
factor set on $G_1/N_1\cong H_1/M_1$. Assume also that the following
conditions hold.
\begin{enumerate}
  \item For every $N\leq V\leq G$, induction gives bijections
  \[
    \Irr(V\cap G_1\mid\eta)\longrightarrow\Irr(V\mid\psi),
    \qquad
    \Irr(V\cap H_1\mid\eta')\longrightarrow\Irr(V\cap H\mid\psi').
  \]
  \item We have
  \[
    \htc(\psi)-\htc(\eta)=\htc(\psi')-\htc(\eta'),
  \]
  and some defect group $Q$ of $\bl(\psi')$ satisfies
  $\N_N(Q)\leq M$.
  \item For every $a\in\Gamma$, there is $m\in M$ such that
  $a_0=a(m^{-1},1)=(h_0,\sigma)$ normalizes
  $G_1,H_1,N_1,M_1$, fixes $\eta,\eta'$, and satisfies
  \begin{equation}\label{eq:3-3}
    \RR^{a_0}=\kappa_{a_0}A_{a_0}^{-1}\RR A_{a_0},\qquad
    (\RR')^{a_0}
       =\kappa_{a_0}|_{H_1}(A'_{a_0})^{-1}\RR'A'_{a_0},
  \end{equation}
  where $\kappa_{a_0}$ is constant on $N_1$-cosets and equals
  $1$ on $N_1$.
\end{enumerate}
Then the pair induced from $\RR,\RR'$ affords
\begin{equation}\label{eq:3-4}
  (A,N,\psi)_{\HH}\geq_b(B,M,\psi')_{\HH}.
\end{equation}
Neither $[N:N_1]=[M:M_1]$ nor $N_1\unlhd N$ is required.
\end{lemma}

\begin{proof}
Taking ordinary stabilizers in \eqref{eq:3-1} gives
$H=G\cap B$. Since $A=NB$ and $N$ acts trivially on $\psi$,
we have $G=NH$. The given block isomorphism yields
$G_1=N_1H_1$. Moreover, $N_1\unlhd G_1$, and the natural maps
induce isomorphisms
\begin{equation}\label{eq:3-5}
  G_1/N_1\cong G/N\cong H/M\cong H_1/M_1.
\end{equation}
Let $\alpha_0$ be the common factor set on these quotients, and
let $\alpha$ be its inflation to $G$.

Choose right coset representatives $r_1=1,\ldots,r_s$ for $N_1$
in $N$ and $s_1=1,\ldots,s_t$ for $M_1$ in $M$. These also
represent $G_1\backslash G$ and $H_1\backslash H$, respectively.
Define
\begin{equation}\label{eq:3-6}
  \PP(x)_{ij}=
  \begin{cases}
    \RR(r_i x r_j^{-1}),&r_i x r_j^{-1}\in G_1,\\
    0,&\text{otherwise},
  \end{cases}
  \qquad x\in G,
\end{equation}
and
\begin{equation}\label{eq:3-7}
  \PP'(y)_{ij}=
  \begin{cases}
    \RR'(s_i y s_j^{-1}),&s_i y s_j^{-1}\in H_1,\\
    0,&\text{otherwise},
  \end{cases}
  \qquad y\in H.
\end{equation}
The matrices have $s$ and $t$ block rows, respectively, and their
underlying permutations need not agree. If two consecutive blocks
in the first matrix product are nonzero, then
\[
\begin{aligned}
  \RR(r_i x r_j^{-1})\RR(r_j y r_k^{-1})
    &={}\alpha_0\bigl((r_i x r_j^{-1})N_1,
                       (r_j y r_k^{-1})N_1\bigr)
          \RR(r_i xy r_k^{-1})\\
    &={}\alpha(x,y)\RR(r_i xy r_k^{-1}).
\end{aligned}
\]
The second equality follows from $r_i,r_j,r_k\in N$ and
\eqref{eq:3-5}. Thus $\PP$ has factor set $\alpha$. In the
second matrix product, the corresponding scalar is
\[
  \alpha_0\bigl((s_i x s_j^{-1})M_1,
                (s_j y s_k^{-1})M_1\bigr)=\alpha(x,y).
\]
Hence $\PP'$ has factor set $\alpha|_{H\times H}$.
The restrictions of $\PP$ and $\PP'$ to $N$ and $M$ are the
ordinary induced representations affording $\psi$ and $\psi'$,
respectively. Since $\alpha$ is inflated from $G/N$, these
projective representations satisfy the normalization conditions
for associated representations.

We next determine their comparison functions. Fix
$a_0=(h_0,\sigma)$ as in~(iii). Since $h_0$ normalizes
$N,M,N_1,M_1$, there are unique permutations $\pi,\pi'$ and
elements $t_i\in N_1$, $u_j\in M_1$ such that
\begin{equation}\label{eq:3-8}
  h_0^{-1}r_i h_0=t_i r_{\pi(i)},\qquad
  h_0^{-1}s_j h_0=u_j s_{\pi'(j)}.
\end{equation}
Let $\nu_{a_0}$ be the inflation of $\kappa_{a_0}$ to $G$ via
\eqref{eq:3-5}. Define invertible block permutation matrices by
\[
  (B_{a_0})_{i,\pi(i)}=A_{a_0}^{-1}\RR(t_i),\qquad
  (B'_{a_0})_{j,\pi'(j)}=(A'_{a_0})^{-1}\RR'(u_j),
\]
with all other blocks equal to zero. If
$v=r_{\pi(i)}x r_{\pi(j)}^{-1}\in G_1$, then
\[
\begin{aligned}
  \PP^{a_0}(x)_{ij}
    &={}\RR(h_0t_i v t_j^{-1}h_0^{-1})^\sigma\\
    &={}\kappa_{a_0}(t_i v t_j^{-1})
         A_{a_0}^{-1}\RR(t_i v t_j^{-1})A_{a_0}\\
    &={}\nu_{a_0}(x)A_{a_0}^{-1}\RR(t_i)
         \PP(x)_{\pi(i),\pi(j)}\RR(t_j)^{-1}A_{a_0}.
\end{aligned}
\]
Here $\kappa_{a_0}$ is constant on $N_1$-cosets, and the factor
set equals $1$ whenever one argument belongs to $N_1$. If
$v\notin G_1$, the corresponding blocks on both sides vanish.
Consequently,
\[
  \PP^{a_0}=\nu_{a_0}B_{a_0}\PP B_{a_0}^{-1}.
\]
On the local side, expansion using \eqref{eq:3-8} gives the
nonzero block
\[
\begin{aligned}
  &\kappa_{a_0}\bigl(u_i s_{\pi'(i)}y
                         s_{\pi'(j)}^{-1}u_j^{-1}\bigr)
    (A'_{a_0})^{-1}\RR'(u_i)\\[-2pt]
  &\hspace{5em}{}
    \cdot\PP'(y)_{\pi'(i),\pi'(j)}\RR'(u_j)^{-1}A'_{a_0}.
\end{aligned}
\]
By \eqref{eq:3-5}, the scalar is $\nu_{a_0}(y)$. Therefore
\begin{equation}\label{eq:3-9}
  (\PP')^{a_0}
    =\nu_{a_0}|_H B'_{a_0}\PP'(B'_{a_0})^{-1}.
\end{equation}
Thus equality of the comparison functions does not require the
same number of cosets on the two sides.

Write $a=a_0(m,1)$, where $m\in M\leq N$. The associated
representations satisfy
\[
  \PP^{(m,1)}=\PP(m)\PP\PP(m)^{-1},\qquad
  (\PP')^{(m,1)}=\PP'(m)\PP'\PP'(m)^{-1}.
\]
Also, $\nu_{a_0}^{(m,1)}=\nu_{a_0}$ because $\nu_{a_0}$ is
inflated from $G/N$. Applying $(m,1)$ to the preceding identities
gives
\[
\begin{aligned}
  \PP^a
    &={}\nu_{a_0}\bigl(B_{a_0}\PP(m)\bigr)
          \PP\bigl(B_{a_0}\PP(m)\bigr)^{-1},\\
  (\PP')^a
    &={}\nu_{a_0}|_H\bigl(B'_{a_0}\PP'(m)\bigr)
          \PP'\bigl(B'_{a_0}\PP'(m)\bigr)^{-1}.
\end{aligned}
\]
Hence the mixed comparison functions agree on $H$. The resulting
function is independent of the choice of $m$: two choices give
normalized intertwining identities whose restrictions to the
irreducible representation $\PP_N$ differ by a scalar, by
Schur's lemma. Substitution into the identities on $G$ then shows
that the functions agree pointwise. The same argument applies to
$\PP'$.

It remains to verify the block and central conditions for this
pair. Conditions~(i) and~(ii) are the hypotheses of
\cite[Theorem~3.14]{NS}, and \eqref{eq:3-6} and
\eqref{eq:3-7} are the representations used in its proof.
We recall the argument to identify the resulting block
correspondence.

Let $D_1$ be a common defect group of $\bl(\eta)$ and
$\bl(\eta')$. The block results for irreducible induction,
together with transitivity of block induction, give
\[
  \bl(\eta')^M=\bl(\psi'),\qquad
  \bl(\eta)^N=\bl(\psi),\qquad
  \bl(\psi')^N=\bl(\psi).
\]
We may choose defect groups $D_1\leq Q\leq\widetilde Q$,
where $Q$ is a defect group of $\bl(\psi')$ and
$\widetilde Q$ is a defect group of $\bl(\psi)$. The degree
formula for induction gives
\[
  |Q:D_1|=p^{\htc(\psi')-\htc(\eta')},\qquad
  |\widetilde Q:D_1|=p^{\htc(\psi)-\htc(\eta)}.
\]
Condition~(ii) therefore implies $Q=\widetilde Q$. Conjugation
of $Q$ within $M$ preserves $\N_N(Q)\leq M$, so this
normalizer condition can be maintained. The existence of these
block inductions and defect group inclusions is the first part of
the proof of \cite[Theorem~3.14]{NS}.

Let $x\in\C_G(N)$, and choose an extension of $\psi$ to
$\langle N,x\rangle$. Since $x$ is central in this group, the
extension has nonzero value at $x$. By~(i), it is induced from
$\langle N,x\rangle\cap G_1$. Thus $x$ is conjugate within
$\langle N,x\rangle$ to an element of that subgroup. As $x$
is central, it follows that $x\in G_1$, and hence
$x\in\C_{G_1}(N_1)\leq H_1$. The matrices $\RR(x)$ and
$\RR'(x)$ are scalar with the same scalar. Since $x$
centralizes every $r_i$ and $s_j$, the induced matrices have
that same scalar. Finally, $\C_A(N)\leq G$, so
$\C_A(N)\leq B$.

Let $N\leq V\leq G$, and set
\[
  V_1=V\cap G_1,\qquad V'=V\cap H,\qquad V'_1=V\cap H_1.
\]
Take $\xi_1\in\Irr(V_1\mid\eta)$ and set $\xi=\xi_1^V$.
If $\tau_{V_1}^{(1)}$ is the correspondence afforded by
$\RR,\RR'$, then the correspondence afforded by their induced
pair satisfies
\begin{equation}\label{eq:3-10}
  \tau_V(\xi)=\bigl(\tau_{V_1}^{(1)}(\xi_1)\bigr)^{V'}.
\end{equation}
To see this, take the projective representation $\mathcal S$ on
$V_1/N_1\cong V/N$ used in the tensor construction. In each
nonzero block, $\mathcal S(r_i x r_j^{-1})$ is the value at
$x$ of its inflation from $V/N$. Tensoring with $\mathcal S$
therefore commutes with the induction matrix. The same argument
applies locally because $s_j\in M\leq N$. This proves
\eqref{eq:3-10} for the fixed pair $\PP,\PP'$.

Set $\xi'_1=\tau_{V_1}^{(1)}(\xi_1)$. The given block
condition is $\bl(\xi'_1)^{V_1}=\bl(\xi_1)$. Applying
block induction along $V'_1\leq V_1\leq V$ and
$V'_1\leq V'\leq V$ gives
\[
\begin{aligned}
  \bl\bigl(\tau_V(\xi)\bigr)^V
    &={}\bigl(\bl(\xi'_1)^{V'}\bigr)^V
     =\bigl(\bl(\xi'_1)^{V_1}\bigr)^V\\
    &={}\bl(\xi_1)^V=\bl(\xi).
\end{aligned}
\]
These inductions are defined by the common defect group and
normalizer conditions, as in the final part of the proof of
\cite[Theorem~3.14]{NS}. Thus the same pair satisfies the block
conditions and all mixed comparison conditions. Together with
\eqref{eq:3-1} and the group conditions, this proves
\eqref{eq:3-4}.
\end{proof}

\subsection{Gluing over a normal subgroup}

\begin{theorem}\label{thm:3-2}
Let $K\unlhd A$, let $H\leq A$, and set $M=K\cap H$.
Let $D_0$ be a $p$-subgroup of $M$ such that
\[
  A=KH,\qquad H=M\N_A(D_0).
\]
Suppose that there is an $\N_A(D_0)\times\HH$-equivariant
bijection
\[
  \Lambda:\Irr_0(K\mid D_0)\longrightarrow\Irr_0(M\mid D_0)
\]
such that every pair $\theta,\theta'=\Lambda(\theta)$ satisfies
\begin{equation}\label{eq:3-11}
  (A_{\theta^{\HH}},K,\theta)_{\HH}
    \geq_b(H_{(\theta')^{\HH}},M,\theta')_{\HH}.
\end{equation}
Let $K\leq J\unlhd A$, and let $Q\leq J$ be a $p$-subgroup
with $Q\cap K=D_0$. Then there is an
$\N_H(Q)\times\HH$-equivariant bijection
\begin{equation}\label{eq:3-12}
  \Omega_J:\Irr_0(J\mid Q)\longrightarrow\Irr_0(J\cap H\mid Q)
\end{equation}
such that every pair $\psi,\psi'=\Omega_J(\psi)$ satisfies
\[
  (A_{\psi^{\HH}},J,\psi)_{\HH}
    \geq_b(H_{(\psi')^{\HH}},J\cap H,\psi')_{\HH}.
\]
\end{theorem}

\begin{proof}
Set $L=J\cap H$. Then
\begin{equation}\label{eq:3-13}
  J=KL,\qquad \N_A(D_0)\leq H,\qquad
  \N_J(Q)\leq\N_J(D_0)\leq L.
\end{equation}
In particular, $Q\leq L$. The inner action of $M$ fixes the
characters on both sides of $\Lambda$. Since
$H=M\N_A(D_0)$, the map $\Lambda$ is also
$H\times\HH$-equivariant.

Choose representatives of the $\N_A(D_0)\times\HH$-orbits on
$\Irr_0(K\mid D_0)$. These also represent the
$H\times\HH$-orbits, since $M$ acts trivially. For each
representative $\theta$, fix a pair affording
\eqref{eq:3-11}. Transport both representations by the same
group and Galois actions to define families of correspondences
$\tau^{(\theta)}$ for all $\theta$. If two transporting
elements give the same $\theta$, their difference belongs to
its mixed stabilizer. Equality of comparison functions and the
tensor construction give
\begin{equation}\label{eq:3-14}
  \tau_{W^h}^{(\theta^a)}(\xi^a)
    =\bigl(\tau_W^{(\theta)}(\xi)\bigr)^a,
  \qquad a=(h,\sigma),\quad K\leq W\leq A_\theta.
\end{equation}
More explicitly, if
$\xi=\tr(\mathcal S\otimes\PP_W)$, then both transformed
tensor constructions use the factor $\mu_a\mathcal S^a$.
Thus \cite[Lemma~1.9(a)]{NSV}, applied to the fixed pair,
proves independence of the transporting choices.

Let $\psi\in\Irr_0(J\mid Q)$. By
\cite[Proposition~2.5(a),(e),(f)]{NS}, there is
$\theta\in\Irr_0(K\mid D_0)$ whose Clifford correspondent
satisfies
\[
  T=J_\theta,\qquad
  \eta\in\Irr_0(T\mid Q,\theta),\qquad \eta^J=\psi.
\]
We call such a constituent \emph{$Q$-adapted} to $\psi$.
Set $U=T\cap H=L_{\theta'}$, and define
\begin{equation}\label{eq:3-15}
  \eta'=\tau_T^{(\theta)}(\eta),\qquad \psi'=(\eta')^L.
\end{equation}
Since $J\unlhd A$ and $A_\theta\unlhd A_{\theta^{\HH}}$,
we have $T\unlhd A_{\theta^{\HH}}$. Also $T\leq A_\theta$.
We may therefore apply Theorem~\ref{thm:2-7} inside $A_{\theta^{\HH}}$,
with intermediate normal subgroup $T$. It gives a fixed pair
affording the block $\HH$-triple relation for $\eta,\eta'$.
Moreover, $\eta'$ has height zero, and its block has the
specified defect group $Q$. To verify the latter assertion,
take the common defect group $Q_1$ supplied by homogeneous
transfer, with $Q_1\cap K=D_0$. If $Q_1^t=Q$ for $t\in T$,
then $t\in\N_T(D_0)\leq U$. Thus $Q$ is also a defect group
of $\bl(\eta')$.

Since $U=L_{\theta'}$ and $\eta'$ lies over $\theta'$,
the induced character $\psi'$ is irreducible. The equality
$\bl(\eta')^T=\bl(\eta)$ and the two subgroup chains give
\begin{equation}\label{eq:3-16}
  \bl(\psi')^J=\bl(\psi).
\end{equation}
Now $\bl(\eta')^L=\bl(\psi')$ has a defect group
$R\geq Q$, whereas \eqref{eq:3-16} places $R$ in a
$J$-conjugate of $Q$. Hence $R=Q$. This is the block
induction argument in the proof of
\cite[Proposition~4.7(b)]{NS}. Furthermore,
\[
  \psi'(1)_p|Q|=[L:U]_p\eta'(1)_p|Q|=|L|_p,
\]
so $\psi'$ has height zero.

Suppose that $\theta_1$ is another $Q$-adapted constituent,
with Clifford correspondent $\eta_1$. By
\cite[Proposition~2.5(f)]{NS}, there is $x\in\N_J(Q)$ such
that $(\theta_1,\eta_1)=(\theta^x,\eta^x)$.
Equation~\eqref{eq:3-14} replaces the local correspondent by
$(\eta')^x$. Since $x\in L$, its induction to $L$ remains
$\psi'$. Thus \eqref{eq:3-12} is well-defined, and the same
identity proves $\N_H(Q)\times\HH$-equivariance.

We next prove bijectivity. Given $\psi'\in\Irr_0(L\mid Q)$,
choose a $Q$-adapted constituent $\theta'$ on $M$ and its
height-zero Clifford correspondent $\eta'$. Let
$\theta=\Lambda^{-1}(\theta')$ and
$\eta=(\tau_T^{(\theta)})^{-1}(\eta')$. Equivariance of
$\Lambda$ identifies the ordinary inertia groups, and
$\eta^J$ is irreducible. Since $\N_J(Q)\leq L$, the Brauer
correspondent $\bl(\psi')^J$ is defined and has defect group
$Q$. Block induction along $U\leq T\leq J$ and
$U\leq L\leq J$ identifies it with $\bl(\eta)^J$.
Thus $\bl(\eta^J)$ has defect group $Q$, and the degree
formula shows that $\eta^J$ has height zero. This constructs a preimage.
If two characters have the same local image, their corresponding
constituents on $M$ are $L$-conjugate. Using
\eqref{eq:3-14}, we may choose the same constituent on both
sides. Uniqueness of the local Clifford correspondent and
injectivity of $\tau_T^{(\theta)}$ then give the same
Clifford correspondent on $T$. Induction to $J$ proves injectivity.

For each corresponding pair, we claim that
\begin{equation}\label{eq:3-17}
  (H\times\HH)_\psi=(H\times\HH)_{\psi'}.
\end{equation}
Suppose that $a=(h,\sigma)$ fixes $\psi$. Then $\theta^a$
is a constituent of $\psi_K$, so $\theta^a=\theta^j$ for
some $j\in J$. Since $J=KL$, there is $l\in L$ such that
$\theta^a=\theta^l$. Hence $a_0=a(l^{-1},1)$ fixes
$\theta,\psi$ and normalizes $T$. Clifford uniqueness gives
$\eta^{a_0}=\eta$, and \eqref{eq:3-14} gives
$(\eta')^{a_0}=\eta'$. Induction shows that $a$ fixes
$\psi'$. Conversely, if $a$ fixes $\psi'$, use a
constituent $\theta'$ on $M$ to choose $l\in L$ such that
$a_0=a(l^{-1},1)$ fixes $\theta'$. Equality of the given
mixed stabilizers implies that $a_0$ fixes $\theta$.
Local Clifford uniqueness and injectivity of
$\tau_T^{(\theta)}$ then show that it fixes $\eta$.
Induction gives $\psi^a=\psi$, proving \eqref{eq:3-17}.

Write $A^*=A_{\psi^{\HH}}$ and
$B^*=H_{(\psi')^{\HH}}$. Since $A=KH$ and $K\leq J$,
equation~\eqref{eq:3-17} gives $A^*=JB^*$ and
$J\cap B^*=L$. Set
\[
\begin{gathered}
  G=(A^*)_\psi,\qquad H^*=(B^*)_{\psi'},\\
  G_1=G_\theta=G_\eta,\qquad
  H_1=(H^*)_{\theta'}=(H^*)_{\eta'}.
\end{gathered}
\]
Apply Lemma~\ref{lem:3-1} with $A^*,B^*,J,L,H^*$ in
place of $A,B,N,M,H$, respectively. Thus $N_1=T$ and
$M_1=U$. Homogeneous transfer gives \eqref{eq:3-2} and
equality of the comparison functions for $a_0$. Clifford
theory gives $G=JG_1$ and $H^*=LH_1$.

For completeness, we verify condition~(i) for every
$J\leq V\leq G$. We have $K\unlhd V$ and
$V_\theta=V\cap G_1$. If $\xi\in\Irr(V\mid\psi)$,
then $\xi_J$ is a multiple of $\psi$. Its Clifford
correspondent over $\theta$ restricts to $T$ with no
constituent other than $\eta$, because $\eta$ is the
unique Clifford correspondent of $\psi$ over $\theta$.
Conversely, every character in $\Irr(V_\theta\mid\eta)$
lies over $\theta$, so it induces irreducibly to $V$, and
its restriction to $J$ lies over $\eta^J=\psi$.
This proves the first bijection in~(i). Locally, apply the same
Clifford uniqueness argument to the normal subgroup
$M=K\cap H$ of $V\cap H^*$. Since $L\leq V\cap H^*$
and $U=L_{\theta'}$, it identifies the inductions of
$\Irr(V\cap H_1\mid\eta')$ with
$\Irr(V\cap H^*\mid\psi')$. All four characters
$\psi,\eta,\psi',\eta'$ have height zero, and
$\N_J(Q)\leq L$, so~(ii) holds. The element $l\in L$
constructed above verifies~(iii). Lemma~\ref{lem:3-1}
therefore proves the asserted block $\HH$-triple relation.
\end{proof}

\subsection{Gluing centralizing correspondences}

\begin{theorem}\label{thm:3-3}
Let $K\unlhd A$, let $K\leq X\unlhd A$, and let
$D\leq X$ be a $p$-subgroup. Write
\[
  D_0=K\cap D,\qquad C=\N_A(D),\qquad
  Y=K\N_X(D),\qquad E=KC,
\]
and define
\[
  \mathcal F=\{\zeta\in\Irr_0(K\mid D_0)\mid
                    \zeta^d=\zeta\text{ for every }d\in D\},
  \qquad \Gamma=C\times\HH.
\]
For each representative $\zeta$ of the $\Gamma$-orbits on
$\mathcal F$, write
\[
  A^*_\zeta=A_{\zeta^{\HH}},\qquad T_\zeta=X_\zeta,
  \qquad U_\zeta=K\N_{T_\zeta}(D),\qquad
  E_\zeta=K\N_{A^*_\zeta}(D).
\]
Suppose that there is a $\Gamma_\zeta$-equivariant bijection
\begin{equation}\label{eq:3-18}
  F_\zeta:\Irr_0(T_\zeta\mid D,\zeta)
       \longrightarrow\Irr_0(U_\zeta\mid D,\zeta)
\end{equation}
such that every pair $\eta,\eta'=F_\zeta(\eta)$ satisfies
\begin{equation}\label{eq:3-19}
  \bigl((A^*_\zeta)_{\eta^{\HH}},T_\zeta,\eta\bigr)_{\HH}
    \geq_b
  \bigl((E_\zeta)_{(\eta')^{\HH}},U_\zeta,\eta'\bigr)_{\HH}.
\end{equation}
The sets in \eqref{eq:3-18} use the specified defect group
$D$, and one pair of projective representations affords both
the block and mixed comparison conditions in
\eqref{eq:3-19}. Then there is a $C\times\HH$-equivariant
bijection
\begin{equation}\label{eq:3-20}
  F:\Irr_0(X\mid D)\longrightarrow\Irr_0(Y\mid D)
\end{equation}
such that every pair $\chi,\chi'=F(\chi)$ satisfies
\begin{equation}\label{eq:3-21}
  (A_{\chi^{\HH}},X,\chi)_{\HH}
    \geq_b(E_{(\chi')^{\HH}},Y,\chi')_{\HH}.
\end{equation}
\end{theorem}

\begin{proof}
We have $Y\unlhd E$, $X\cap E=Y$, and $D\leq Y$.
For $\zeta\in\mathcal F$, we have $D\leq T_\zeta$ and
\begin{equation}\label{eq:3-22}
  Y_\zeta=K(\N_X(D))_\zeta
    =K\N_{X_\zeta}(D)=U_\zeta.
\end{equation}
Transport the maps $F_\zeta$ from the chosen representatives
to all of $\mathcal F$. Equivariance under the mixed
stabilizer makes the resulting maps independent of the
transporting choices and gives
\begin{equation}\label{eq:3-23}
  F_{\zeta^a}(\eta^a)=F_\zeta(\eta)^a
  \qquad(a\in\Gamma).
\end{equation}
We transport the associated pairs at the same time; their
comparison functions remain equal on each mixed stabilizer.

Let $\chi\in\Irr_0(X\mid D)$. By
\cite[Proposition~2.5(a),(e),(f)]{NS}, choose a $D$-adapted
constituent $\zeta\in\mathcal F$ with Clifford correspondent
\[
  \eta\in\Irr_0(T_\zeta\mid D,\zeta),\qquad \eta^X=\chi.
\]
Define
\begin{equation}\label{eq:3-24}
  \eta'=F_\zeta(\eta),\qquad F(\chi)=(\eta')^Y.
\end{equation}
This induction is irreducible by \eqref{eq:3-22}. If
$\zeta_1$ is another adapted constituent, with correspondent
$\eta_1$, then \cite[Proposition~2.5(f)]{NS} gives
$x\in\N_X(D)$ such that
$(\zeta_1,\eta_1)=(\zeta^x,\eta^x)$.
Equation~\eqref{eq:3-23} replaces $\eta'$ by $(\eta')^x$.
Since $x\in Y$, induction to $Y$ leaves the character
unchanged. Thus \eqref{eq:3-24} is well-defined. For
$a\in\Gamma$, the constituent $\zeta^a$ is adapted to
$\chi^a$, with Clifford correspondent $\eta^a$. Applying
\eqref{eq:3-23} and inducing yields $F(\chi^a)=F(\chi)^a$.

Write $T=T_\zeta$ and $U=U_\zeta$. The hypothesis gives
$\bl(\eta')^T=\bl(\eta)$. Block induction along
$U\leq T\leq X$ and $U\leq Y\leq X$, as in the proof
of \cite[Proposition~7.4]{NS}, gives
\begin{equation}\label{eq:3-25}
  \bl\bigl((\eta')^Y\bigr)^X=\bl(\eta^X)=\bl(\chi).
\end{equation}
Choose a defect group $R\geq D$ of
$\bl((\eta')^Y)$. By \eqref{eq:3-25}, it is contained
in an $X$-conjugate of $D$, so $R=D$. Since $\eta'$
has height zero,
\[
  \bigl((\eta')^Y\bigr)(1)_p|D|
     =[Y:U]_p\eta'(1)_p|D|=|Y|_p.
\]
Hence $F(\chi)\in\Irr_0(Y\mid D)$ for every
$\chi\in\Irr_0(X\mid D)$.

Conversely, let $\chi'\in\Irr_0(Y\mid D)$. Choose a
$D$-adapted constituent $\zeta\in\mathcal F$ on $K$
and a Clifford correspondent
$\eta'\in\Irr_0(Y_\zeta\mid D,\zeta)$ such that
$(\eta')^Y=\chi'$. By \eqref{eq:3-22}, we may set
$\eta=F_\zeta^{-1}(\eta')$, and $\chi=\eta^X$ is
irreducible. Since $\N_X(D)\leq Y$, Brauer's first main
theorem shows that $\bl(\chi')^X$ is defined and has
defect group $D$. The two block induction chains identify
it with $\bl(\eta)^X=\bl(\chi)$. Since $\eta$ has height zero and $\bl(\eta)$ has defect
group $D$, the degree formula for induction shows that
$\chi$ has height zero. This
proves surjectivity.

Suppose that $F(\chi_1)=F(\chi_2)$, and choose adapted
pairs $(\zeta_i,\eta_i)$. Then
$\eta'_i=F_{\zeta_i}(\eta_i)$ are $D$-adapted Clifford
correspondents of the same local character. By
\cite[Proposition~2.5(f)]{NS}, there is
$x\in\N_Y(D)=\N_X(D)$ such that
$(\zeta_2,\eta'_2)=(\zeta_1^x,(\eta'_1)^x)$.
Equation~\eqref{eq:3-23} and injectivity of $F_{\zeta_2}$
give $\eta_2=\eta_1^x$. Induction to $X$ yields
$\chi_1=\chi_2$. Thus $F$ is a bijection.

Since $E=KC$ and $K$ acts trivially on irreducible characters
of both $X$ and $Y$, equivariance gives
\begin{equation}\label{eq:3-26}
  (E\times\HH)_\chi=(E\times\HH)_{\chi'}.
\end{equation}
Moreover, $D$ is a defect group of $\bl(\chi)$, and
$\HH$ preserves the actual defect groups by Lemma~\ref{lem:2-1}.
For every $g\in A_{\chi^{\HH}}$, we can therefore choose
$x\in X$ such that $gx\in\N_A(D)$. Inner conjugation
fixes $\chi$, so $gx$ still stabilizes $\chi^{\HH}$.
Together with \eqref{eq:3-26}, this gives
\begin{equation}\label{eq:3-27}
  A_{\chi^{\HH}}=XE_{(\chi')^{\HH}},\qquad
  X\cap E_{(\chi')^{\HH}}=Y.
\end{equation}

Fix the adapted characters $\zeta,\eta,\eta'$, and let
$a=(h,\sigma)\in(E\times\HH)_\chi$. Write $h=nk$ with
$n\in C$ and $k\in K$. Then $a_1=a(k^{-1},1)$ still
fixes $\chi$, and its group component normalizes $D$.
Thus $(\zeta^{a_1},\eta^{a_1})$ and $(\zeta,\eta)$ are
both $D$-adapted constituent--correspondent pairs for
$\chi$. By \cite[Proposition~2.5(f)]{NS}, there is
$u\in\N_X(D)$ such that
\[
  \zeta^{a_1}=\zeta^u,\qquad \eta^{a_1}=\eta^u.
\]
Set $a_0=a_1(u^{-1},1)=a((uk)^{-1},1)$. Its group
component normalizes $D$, it fixes $\zeta,\eta$, and
the correcting element $uk$ belongs to $Y$. By
\eqref{eq:3-23}, it also fixes $\eta'$. Consequently,
$a_0$ belongs to the local mixed stabilizer in
\eqref{eq:3-19}, where the fixed associated pair has
equal comparison functions.

Apply Lemma~\ref{lem:3-1} with
\[
  G=A_\chi,\qquad H=E_{\chi'},\qquad
  G_1=(A_\chi)_\zeta,\qquad H_1=(E_{\chi'})_\zeta.
\]
Clifford uniqueness identifies $G_1$ with the ordinary
stabilizer of $\eta$ in $A^*_\zeta$, and $H_1$ with
the ordinary stabilizer of $\eta'$ in $E_\zeta$.
For the latter equality, let $h\in(E_{\chi'})_\zeta$
and write $h=kn$ with $k\in K$ and $n\in\N_A(D)$.
Since $k$ fixes $\zeta$, so does $n$; hence
$h\in E_\zeta$. Clifford uniqueness now implies that
$h$ fixes $\eta'$. Conversely, $\eta'_K$ is a positive
multiple of $\zeta$, and induction commutes with
conjugation, giving the reverse inclusion. We also have
\[
  G=XG_1,\qquad H=YH_1,\qquad
  G_1\cap X=T_\zeta,\qquad H_1\cap Y=U_\zeta.
\]
The ordinary part of \eqref{eq:3-19} is precisely the
block isomorphism required in Lemma~\ref{lem:3-1}.

Let $X\leq V\leq A_\chi$. We have $K\unlhd V$, and
$V$-invariance of $\chi$ shows that the $V$-orbit of
its constituents on $K$ is already the $X$-orbit.
Hence $V=XV_\zeta$. If $\xi\in\Irr(V\mid\chi)$,
its Clifford correspondent $\xi_\zeta$ over $\zeta$
restricts to $T_\zeta$ with only $\eta$ as a
constituent: $\xi_X$ is a multiple of $\chi$, and
$\eta$ is the unique Clifford correspondent of $\chi$
over $\zeta$. Conversely, a character
$\xi_1\in\Irr(V_\zeta\mid\eta)$ lies over $\zeta$,
so $\xi_1^V$ is irreducible and lies over
$\eta^X=\chi$. Thus induction gives a bijection
\[
  \Irr(V_\zeta\mid\eta)\longrightarrow\Irr(V\mid\chi).
\]
Locally, $K\unlhd V\cap E$, $Y\leq V\cap E$, and
$\chi'$ is invariant in $V\cap E$. The Clifford
correspondent over $\zeta$ restricts to $Y_\zeta$
with only $\eta'$ as a constituent; conversely,
induction is irreducible and lies over
$(\eta')^Y=\chi'$. This proves both bijections in
condition~(i) for every intermediate group.

All four characters $\eta,\eta',\chi,\chi'$ have
height zero, so both height differences vanish. The block $\bl(\chi')$ has the specified defect group
$D$, and $\N_X(D)\leq Y$, proving~(ii). The element
$a_0$ constructed above normalizes
$G_1,H_1,T_\zeta,U_\zeta$: it fixes the relevant
constituents and upper characters, and the Galois
action commutes with the group action. The given
relation yields equality of its comparison functions,
proving~(iii). Lemma~\ref{lem:3-1} and
\eqref{eq:3-27} now give \eqref{eq:3-21}. The
construction uses separate representatives for
$T_\zeta\backslash X$ and $U_\zeta\backslash Y$,
and does not require equality of the two Clifford
indices.
\end{proof}


\end{document}